\documentclass[11pt]{amsart}
\usepackage[utf8]{inputenc}
\usepackage{lmodern}
\usepackage{amsmath, amsthm, amssymb, amsfonts,bm}
\usepackage[normalem]{ulem}
\usepackage{hyperref}
\usepackage{varwidth}
\usepackage{adjustbox}
\usepackage{enumerate}

\usepackage{verbatim} 
\usepackage{longtable}

\usepackage{mathtools}

\usepackage{enumitem}
\usepackage{tikz}
\usetikzlibrary{decorations.pathmorphing}
\tikzset{snake it/.style={decorate, decoration=snake}}
\usepackage{tikz-cd}
\usetikzlibrary{positioning}

\usepackage{caption}

\usepackage{tikz-cd}
\usetikzlibrary{arrows}
\usepackage{ytableau}
\usepackage{mathdots}
\theoremstyle{plain}

\newtheorem{theorem}{Theorem}[section]

\newtheorem{definition}[theorem]{Definition}
\newtheorem{assumption}[theorem]{Assumption}
\newtheorem{example}[theorem]{Example}
\newtheorem{proposition}[theorem]{Proposition}
\newtheorem{prop}[theorem]{Proposition}
\theoremstyle{definition}
\newtheorem{exmp}[theorem]{Examples}

\theoremstyle{remark}
\newtheorem{rem}[theorem]{Remark}
\newtheorem{lem}[theorem]{Lemma}
\newtheorem{ques}[theorem]{Question}

\newtheorem{lemma}[theorem]{Lemma}
\newtheorem{corollary}[theorem]{Corollary}
\newtheorem{remark}[theorem]{Remark}

\newcommand{\C}{\mathbb C}

\newcommand{\Pic}{\mathrm{Pic}}

\newcommand{\End}{\mathcal E nd}
\newcommand{\GL}{\mathrm{GL}}
\newcommand{\id}{\mathrm{id}}
\newcommand{\gr}{\operatorname{gr}}
\newcommand{\rk}{\operatorname{rk}}

\newcommand{\Coker}{\operatorname{Coker}}

\newcommand{\HH}{\mathbb H}

\newcommand{\ov}[1]{\overline{#1}}

\newcommand{\Flag}{\mathrm{Flag}}
\newcommand{\Fr}{\operatorname{Fr}}

\newcommand{\PEnd}{\mathcal P\!\End}
\newcommand{\SPEnd}{\mathcal{SP}\!\End}

\newcommand{\cO}{\mathcal O}

\newcommand{\cK}{\mathcal K}
\newcommand{\cM}{\mathcal M}

\newcommand{\Higgs}{\mathrm{Higgs}}

\newcommand{\tr}{\mathrm{tr}}
\newcommand{\nmd}{\mathcal{N}(\vec{m})}
\newcommand{\pardeg}{\operatorname{pardeg}}
\newcommand{\parmu}{\operatorname{par\mu}}

\DeclareFontFamily{OT1}{rsfs}{}
\DeclareFontShape{OT1}{rsfs}{n}{it}{<-> rsfs10}{}
\DeclareMathAlphabet{\curly}{OT1}{rsfs}{n}{it}

\def\C{\mathbb{C}}

\def\O{\mathcal{O}}

\def\a{\alpha}

\def\m{\mu}

\def\m{\mathcal}

\def\bb{\mathbb}

\def\deg{\textrm{deg}}

\def\rk{\textrm{rank}}
\def\dim{\mathrm{dim}}

\def\ker{\textrm{ker}}

\def\gr{\textrm{gr}}
\def\Pic{\textrm{Pic}}

\def\res{\mathrm{res}}

\def\uxi{{\underline{\xi}}}
\def\vxi{\vec{\xi}}
\def\vm{\vec{m}}
\def\va{\vec{\a}}

\def\bmxi{B(\vm)_{\vxi}}

\def\um{{\underline{m}}}

\def\nm{\mathcal{N}(\vm)}

\newcommand{\pr}{\textrm{pr}}

\newsavebox{\leftbox} \newsavebox{\rightbox}%

\NewDocumentCommand{\lrboxbrace}{s O{\{} O{\}} O{0.05\linewidth} m O{0.5\linewidth} m}{% \lrboxbrace[<lbrace>][<rbrace>][<lwidth>]{<ltext>}[<rwidth>]{<rtext>}
  \begin{lrbox}{\leftbox}% Left box
    \IfBooleanTF{#1}% starred/unstarred
      {\begin{varwidth}{#4}#5\end{varwidth}}
      {\begin{minipage}{#4}#5\end{minipage}}
  \end{lrbox}
  \begin{lrbox}{\rightbox}% Right box
    \IfBooleanTF{#1}% starred/unstarred
      {\begin{varwidth}{#6}#7\end{varwidth}}
      {\begin{minipage}{#6}#7\end{minipage}}
  \end{lrbox}
  \ensuremath{\usebox\leftbox\left\{\,\usebox\rightbox\,\right\}}
}

\usepackage[backend=biber,style=alphabetic,sorting=nyt, giveninits=true,doi=false,isbn=false,url=false,maxbibnames=99]{biblatex} %main package for bibtex
\usepackage{vmargin}
\setpapersize{A4}
\setmarginsrb{27mm}{12mm}{27mm}{12mm}%
{12mm}{10mm}{5mm}{10mm}

\usepackage{tikz}
\usepackage{lmodern}
\usetikzlibrary{decorations.pathmorphing}

\begin{document}
\title{Symplectic leaves of meromorphic Hitchin systems}

\author{Jia Choon Lee and Sukjoo Lee}
\address{Institute of Geometry and Physics, University of Science and Technology of China, 96 Jinzhai Road, Hefei 230026 P.R. China}
\email{jiachoonlee@ustc.edu.cn}

\address{Institute for Basic Science - Center for Geometry and Physics, 79 Jigok-ro 127 beon-gil, Nam-gu, Pohang, Gyeongbuk, KOREA 37673
}
\email{sukjoo216@ibs.re.kr}

\begin{abstract}
    The moduli space of meromorphic Higgs bundles admits a Poisson structure due to the independent work of Bottacin and Markman. In this paper, we revisit the symplectic leaves of this Poisson structure for the tame case. We study the partial compactification of the restricted Hitchin map on the symplectic leaves to an algebraically completely integrable system. In particular, we show that such a partial compactification is realized by the moduli spaces of $\vec{\xi}$-parabolic Higgs bundles. These same moduli spaces also provide a symplectic resolution of the normalization of the closure of the corresponding symplectic leaves. Finally, we discuss connectedness results for the corresponding Betti moduli spaces under the tame non-abelian Hodge correspondence. 
\end{abstract}

\baselineskip=14.5pt
\maketitle

\setcounter{tocdepth}{2} 

\tableofcontents

\section{Introduction}
\subsection{Overview}
Higgs bundles were introduced by Hitchin in his study of the self-duality equations on Riemann surfaces \cite{Hitchin1987SelfDuality}. One of the fundamental features of the moduli space of stable Higgs bundles is that it carries a natural holomorphic symplectic structure, and the Hitchin map defines an algebraically completely integrable system, now known as the Hitchin system \cite{Hitchin1987Stable,DonagiMarkman1996}.

A natural generalization is obtained by allowing the Higgs field to have poles at marked points. Let $C$ be a smooth projective curve and let $D\subset C$ be a reduced effective divisor. For simplicity, only in this overview section, we discuss the case where $D$ consists of a single point $p$. A meromorphic Higgs bundle is a pair $(E,\Phi)$, where $E$ is a rank $r$ vector bundle on $C$ and $\Phi:E\to E\otimes K_C(D)$ is an $\mathcal O_C$-linear map. We write $\Higgs^s(C,D)$ for the moduli space of stable meromorphic Higgs bundles, which is also known as the meromorphic Hitchin system. Unlike the case without poles, $\Higgs^s(C,D)$ does not carry a holomorphic symplectic structure on the whole space. Instead, it carries a holomorphic Poisson structure due to the independent work of Bottacin \cite{Bottacin1995symplectic} and Markman \cite{markmanspectral}.

For a Poisson variety, a basic geometric problem is to understand its symplectic leaves. In the setting of meromorphic Higgs bundles, the residue at the pole provides the relevant local model. The residue $\res_p(\Phi)$ is an element of $\mathfrak{gl}_r$. Identifying $\mathfrak{gl}_r\cong \mathfrak{gl}_r^*$ under the trace pairing, the adjoint orbits in $\mathfrak{gl}_r$  are the symplectic leaves of the standard Lie--Poisson structure and their induced symplectic forms are the Kirillov--Kostant--Souriau forms. This suggests that, on the global moduli space $\Higgs^s(C,D)$, the loci obtained by requiring the residue to lie in a fixed adjoint orbit $\cO$ should play the role of symplectic leaves. We denote these loci by $\Higgs^{\cO,s}(C,D)$. 

\begin{ques}
For an arbitrary adjoint orbit $\cO$, is $\Higgs^{\mathcal O,s}(C,D)$ a symplectic leaf of $\Higgs^s(C,D)$? 
\end{ques}
Recall that a symplectic leaf of a Poisson variety is a maximal connected locally closed smooth subvariety tangent to the characteristic distribution of the Poisson structure. Up to the question of connectedness, this picture is already present in the work of Bottacin \cite[Theorem 4.7.5]{Bottacin1995symplectic} and Markman \cite[Corollary 8.10]{markmanspectral}, especially for symplectic leaves of maximal rank. For the purposes of the present paper, we will need to keep track of the explicit infinitesimal description of $\Higgs^{\cO,s}(C,D)$ in terms of characteristic distribution and to address the connectedness of $\Higgs^{\cO,s}(C,D)$ to ensure that it is a single symplectic leaf.

Once the symplectic leaf is identified, a further question is to understand the Hitchin map restricted to it. It is shown in \cite[Section 4.7] {Bottacin1995symplectic} and \cite[Section 8.3]{markmanspectral} that the Hitchin map on $\Higgs^s(C, D)$ is a Poisson integrable system in a generalized sense. However, this does not immediately imply that its restriction to an arbitrary symplectic leaf is an algebraically completely integrable system in the sense that it is a proper flat Lagrangian fibration with generic fibers abelian varieties. There are two basic issues. First, when the symplectic leaf is not maximal rank, the corresponding spectral curves are typically singular, so the usual description of the generic Hitchin fibers as Jacobians of smooth spectral curves no longer applies directly. Second, if the symplectic leaf is not closed in $\Higgs^s(C, D)$, the restricted Hitchin map may fail to be proper. 
\begin{ques}
    For an arbitrary adjoint orbit $\cO$, can the restricted Hitchin map on $\Higgs^{\cO,s}(C,D)$ be partially compactified to an algebraically completely integrable system? 
\end{ques}
This question is closely related to the "Lax project"  formulated by Boalch in \cite[Section 1]{boalch2018wild}. Boalch proposes to classify integrable systems admitting a Lax representation, meaning that it is isomorphic to a symplectic leaf of a meromorphic Hitchin system, with the Hamiltonians given by restricting the Hitchin map.\footnote{This framework also includes irregular meromorphic Higgs fields, which we do not consider in this paper.} However, Boalch restricts the program to a class of "good" Lax representations, which correspond to closed symplectic leaves of maximal rank in our case, since a Poisson integrable system does not necessarily restrict to an algebraically completely integrable system on every symplectic leaf. Thus the present question can be viewed as complementary to the "good" Lax representations by considering symplectic leaves of arbitrary rank. 

There are natural candidates for such a partial compactification. First, one may take the closure of $\Higgs^{\cO,s}(C,D)$ in $\Higgs^s(C, D)$. Second, one may replace the residue orbit $\cO$ by its closure $\overline{\cO}$ and consider the corresponding moduli space $\Higgs^{\overline{\cO},s}(C,D)$. While these two constructions are not a priori the same, one of our results is that they coincide. However, this partial compactification is generally singular, reflecting the singularities of $\overline{\cO}.$ 

The key point is that the singularities of $\overline{\cO}$ admit a natural resolution by adding compatible flag data. This suggests replacing a meromorphic Higgs bundle with residue in $\overline{\cO}$ by a parabolic Higgs bundle with conditions on the flags determined by $\cO$. This leads us to consider the moduli space of $\vxi$-parabolic Higgs bundles, whose smoothness was established previously in \cite{LeeLee2024,LeeLee2025}, in order to resolve $\Higgs^{\overline{\cO}, s}(C,D)$. Moreover,  these moduli spaces form algebraically completely integrable systems, providing an answer to the partial compactification question above.

As a byproduct, we obtain some topological consequences for certain Betti moduli spaces via Simpson's tame non-abelian Hodge correspondence \cite{simpson-noncompact}. In particular, we obtain connectedness results for moduli spaces of filtered local systems and for irreducible loci of character varieties with prescribed local monodromy.

\subsection{Main results}
Let $C$ be a smooth projective complex curve of genus $g$, and let $D=p_1+\cdots+p_n$ be a reduced effective divisor. We fix a rank $r$ and a degree $d$. To each $p_i$, we assign a conjugacy class $\cO_i$ of $\mathfrak{gl}_r$ and its closure $\ov\cO_i$. We set $\cO:=\prod_{i=1}^n \cO_i$ and $\ov{\cO}:=\prod_{i=1}^n \ov{\cO_i}$. Throughout this article, we always assume that for each $A_i \in \cO_i$, the sum of traces is zero;
\begin{equation}\label{eq:trace-compatibility}
\sum_{i=1}^n\tr(A_i)=0 \quad \text{for }A_i \in \cO_i. 
\end{equation}
Note that if there exists a meromorphic Higgs bundle $(E,\Phi)$ with
$\res_{p_i}(\Phi)\in \cO_i$ for all $i$, then the condition
\eqref{eq:trace-compatibility} holds automatically by the residue theorem applied to $\tr(\Phi)$. As before, we write
$\Higgs^{\cO,s}(C,D)\subset \Higgs^s(C,D)$ and
$\Higgs^{\ov{\cO},s}(C,D)\subset \Higgs^s(C,D)$ for the corresponding stable
loci defined by the orbits $\cO_i$ and their closures $\ov{\cO_i}$, respectively.
We use the analogous notation
$\Higgs^{\cO,ss}(C,D)\subset \Higgs^{\ov{\cO},ss}(C,D)\subset
\Higgs^{ss}(C,D)$ for the semistable moduli spaces.

\begin{theorem}[Theorem~\ref{thm:connected HiggsO}, Theorem~\ref{thm:orbit-stratum-is-leaf}]\label{thm:poisson-closure}
	For a product of conjugacy classes $\cO$ satisfying the condition \eqref{eq:trace-compatibility}, if $\Higgs^{\cO,s}(C,D)$ is non-empty, then it is irreducible and is a symplectic leaf of the Poisson variety $\Higgs^s(C,D)$ whose closure in $\Higgs^{s}(C,D)$ is $\Higgs^{\ov{\cO},s}(C,D)$. In particular, $\Higgs^{\ov{\cO},s}(C,D)$ is irreducible.  
\end{theorem}
The proof of Theorem~\ref{thm:poisson-closure} has an infinitesimal part and a global part. First, we give a self-contained deformation theoretic proof to show that the tangent space of $\Higgs^{\cO, s}(C,D)$ agrees with the characteristic distribution of the Bottacin--Markman Poisson structure. For the global part, in order to identify the closure of $\Higgs^{\cO, s}(C, D)$ and prove its connectedness, we use a comparison with the corresponding moduli space of $\vxi$-parabolic Higgs bundles.

For each $p_i\in D$, we fix a length $\ell_i$ partition $\um_i=(m_{i,1},\dots,m_{i,\ell_i})$ of $r$ and a tuple $\uxi_i=(\xi_{i,1},\dots,\xi_{i,\ell_i})\in \C^{\ell_i}$. We write $\vec m=(\um_1,\dots,\um_n)$ and $\vec\xi=(\uxi_1,\dots,\uxi_n)$. Then a parabolic Higgs bundle is a triple $(E,E_D^\bullet, \Phi)$ where $(E,\Phi)$ is a meromorphic Higgs bundle and $E_D^\bullet$ is a collection of flags $E_{p_i}^\bullet$ of length $\ell_i$ which are compatible with the restriction $\Phi|_{p_i}$. In particular, we call this triple $(E,E_D^\bullet, \Phi)$ a $\vxi$-parabolic Higgs bundle if the residue $\res_{p_i}(\Phi)$ acts by scalar multiplication by $\xi_{i,a}$ on each associated graded piece $E^{a-1}_{p_i}/E^a_{p_i}$ for $a=1, \cdots, \ell_i$; see Definition \ref{def:xi-par}.

For a system of weights $\vec\alpha$, we write $
\Higgs^{\text{par}}(C,D,\vec m,\vec\alpha)$ (resp. $\Higgs^{\vxi\text{-par}}(C,D,\vec m,\vec\alpha)$)
for the moduli space of $\vec{\alpha}$-semistable parabolic (resp. $\vxi$-parabolic) Higgs bundles of type $\vec m$. By construction, $\Higgs^{\vxi\text{-par}}(C,D,\vec m,\vec\alpha)$ is closed in $\Higgs^{\text{par}}(C,D,\vec m,\vec\alpha)$. We say $\vec\alpha$ is \textit{generic} if $\vec\alpha$-semistability implies $\vec\alpha$-stability. 
\begin{assumption}\label{assmp:genericity}
    Throughout this article, we assume the existence of generic weights $\vec\alpha$ and fix one such choice.
\end{assumption}

\begin{remark}\label{rem:assump}
    This assumption holds when the numerical data $(r,d,\vm)$ is primitive; i.e. $\mathrm{gcd}(r,d,\{m_{i,a}\})=1$. 
\end{remark}

 The moduli space $\Higgs^{\vec\xi\text{-par}}(C,D,\vec m,\vec\alpha)$ has been studied in our previous works \cite{LeeLee2024, LeeLee2025} and there are two results we would like to recall:
\begin{enumerate}[label=\textup{(\arabic*)}]
	\item The moduli space $\Higgs^{\vec\xi\text{-par}}(C,D,\vec m,\vec\alpha)$ is smooth \cite[Theorem 2.6]{LeeLee2024}.
    \item We describe the spectral correspondence for $\vxi$-parabolic Higgs bundles \cite[Theorem 1.2]{LeeLee2025}. In particular, the restricted Hitchin map $h:\Higgs^{\vec\xi\text{-par}}(C,D,\vec m,\vec\alpha) \to A$  factors through an affine subspace $B(\vm)_{\vxi}\subset A$ \cite[Corollary 1.4]{LeeLee2025}, and we write 
    \[h(\vm)_{\vxi}: \Higgs^{\vec\xi\text{-par}}(C,D,\vec m,\vec\alpha)\to B(\vm)_{\vxi}\]
    for the induced Hitchin map. The affine subspace $B(\vm)_{\vxi}$ is a linear system of curves contained in a given holomorphic symplectic surface obtained by blowing up the ruled surface $\bb{P}(K_C(D)\oplus\O_C)$. Over the locus of smooth curves in $B(\vm)_{\vxi}$, the  fibers are the corresponding Jacobian varieties. 
    
\end{enumerate} 
\noindent By the work of Logares--Martens \cite{logaresMartens}, the ambient moduli space of parabolic Higgs bundles $\Higgs^{\mathrm{par}}(C,D,\vec m,\vec\alpha)$ carries a natural Poisson
structure. We prove the following analogue of
Theorem~\ref{thm:poisson-closure} (a similar statement can also be found in \cite[Section 3.2.4]{logaresMartens} and \cite[Proposition 7.1]{BGPM2020} for real reductive groups).  
\begin{theorem}[Theorem \ref{thm:symplectic leaf xipar}]
     The moduli space $\Higgs^{\vec\xi\text{-par}}(C,D,\vec m,\vec\alpha)$ is irreducible and is a closed symplectic leaf of the Poisson variety $\Higgs^{\text{par}}(C,D,\vec m,\vec\alpha)$.
\end{theorem}
Together with the spectral correspondence recalled above, the restricted Hitchin map yields the structure of an algebraically completely integrable system on $\Higgs^{\vec\xi\text{-par}}(C,D,\vec m,\vec\alpha)$. 
\begin{theorem}[Theorem \ref{thm:holosym and integrable}]\label{thm:xipar}
     Let the notation be as above. The Hitchin map $h(\vm)_{\vxi}$ forms an algebraically completely integrable system. In other words, $h(\vm)_{\vxi}$ is a proper flat Lagrangian fibration over the affine base $B(\vm)_{\vxi}$ whose generic fibers are abelian varieties. 
    
\end{theorem}
The case $\vxi=0$ corresponds to strongly parabolic Higgs bundles, which has been studied by Scheinost--Schottenloher \cite{ScheinostSchotten95} and Su--Wang--Wen \cite{Suwangwen2022}. It is also generalized to the $G$-parahoric setting by Baraglia--Kamgarpour--Varma \cite{baragliaKamgarpourVarma}.

To relate $\vxi$-parabolic Higgs bundles with meromorphic Higgs bundles with  fixed residue classes, we need the linear algebra fact: For each $i$, there is a one-to-one correspondence between the set of combinatorial data  $(\um_i,\uxi_i)$ and the set of conjugacy classes $\cO_i \in \mathfrak{gl}_r$, up to permutation, in the sense that every element of the orbit $\cO_i$ has a unique
compatible flag of type $(\um_i,\uxi_i)$; see Appendix \ref{sec:local-linear-algebra} for
details. Whenever $\cO_i$ is the conjugacy class determined by $(\um_i,\uxi_i)$, we write
$\cO_i=\cO(\um_i,\uxi_i)$. We write $\cO=\cO(\vm,\vxi)=\prod_{i=1}^n\cO_i $ where
$\cO_i=\cO(\um_i,\uxi_i)$ for every $i$. In this case, we say that
$(\vm,\vxi)$ corresponds to $\cO$. For such corresponding data,
\eqref{eq:trace-compatibility} is equivalent to the following condition:
\begin{equation}\label{eq:Residue condition}
\sum_{i=1}^n\sum_{a=1}^{\ell_i} m_{i,a}\xi_{i,a}=0.
\end{equation}

On the level of the moduli spaces, we observe that there is a well-defined morphism by forgetting the flags
\[
F:\Higgs^{\vec\xi\text{-par}}(C,D,\vec m,\vec\alpha)\longrightarrow \Higgs^{\ov{\cO},ss}(C,D)
\]
for small enough weights $\vec{\alpha}$ (see Lemma \ref{lem:small-weights} and Proposition \ref{prop:well-defined-forgetful}).

\begin{proposition}[Proposition \ref{prop:F properties}]\label{prop:main} The forgetful map $F$ satisfies the following properties.
    \begin{enumerate}
        \item The image of $F$ contains $\Higgs^{\ov\cO,s}(C,D)$ and $F$ is projective over this locus.
        \item $F$ is an isomorphism over $\Higgs^{\cO,s}(C,D) \subset \Higgs^{\ov\cO,s}(C,D)$. Moreover, it is a symplectomorphism. 
        \item $F$ is compatible with Hitchin maps.
    \end{enumerate}
\end{proposition}
Hence, the introduction of $\vxi$-parabolic Higgs bundles serves two purposes. First, via the forgetful map $F$, the irreducibility of $\Higgs^{\vec\xi\text{-par}}(C,D,\vec m,\vec\alpha)$ implies the irreducibility of $\Higgs^{\cO,s}(C,D)$ and $ \Higgs^{\ov\cO,s}(C,D)$. Second, combining Theorem~\ref{thm:xipar} and Proposition~\ref{prop:main} justifies that $\Higgs^{\vec\xi\text{-par}}(C,D,\vec m,\vec\alpha)$ and its Hitchin map $h(\vm)_{\vxi}$ provide a partial compactification of the restricted Hitchin map on $\Higgs^{\cO,s}(C,D)$ to an algebraically completely integrable system. Moreover, under extra assumptions, the forgetful map $F$ resolves the singularities of the normalization of the closure $\Higgs^{\overline{\cO}, ss}(C,D)$.

\begin{theorem}[Theorem \ref{thm:symplectic resolution}]
Assume that $\Higgs^{\cO, s}(C,D)\neq \emptyset$ and $F$ is proper and surjective. 
        Then the normalization of $\Higgs^{\overline{\cO}, ss}(C,D)$ has symplectic singularities and the forgetful map $F$ yields a symplectic resolution. Moreover, this resolution is compatible with the Hitchin maps. 
  
\end{theorem}
Note that when semistability coincides with stability, e.g. when $(r,d)=1$, then the properness and surjectivity required in Theorem~\ref{thm:symplectic resolution} follow from Proposition~\ref{prop:main}. We also note that symplectic resolutions of Higgs moduli spaces without poles have been studied by Kiem--Yoo \cite{KiemYoo} and Tirelli \cite{Tirelli2019SymplecticResolutionsHiggs}. In our setting, we focus primarily on singularities coming from orbit closures of the residues.

\subsection{Applications to Betti moduli spaces}

We finally explain a consequence of our connectedness theorem to the Betti side. Let
$C^\circ=C\setminus D$. A filtered local system on $C^\circ$ is a local system
$\mathbb L$ together with, for each $p_i\in D$, a decreasing left-continuous
filtration $\{\mathbb L_{\widetilde p_i}^{\beta}\}_{\beta\in\mathbb R}$
on a nearby stalk, preserved by the local monodromy around $p_i$. Fix Betti-side
local data, called a residue diagram and denoted by
$P^{\vec\beta,\vec\lambda}_{\vec m}$, consisting of filtration jumps $\vec\beta$,
local monodromy eigenvalues $\vec\lambda$, and Jordan types on the associated
graded pieces; see Section~\ref{sec:Betti} for details. We denote the corresponding
stable moduli space of filtered local systems by $\mathcal M^{s}_{\mathrm{Betti}}(P^{\vec\beta,\vec\lambda}_{\vec m})$.

By Simpson's tame non-abelian Hodge correspondence, recalled in
Theorem~\ref{thm:tame-NAHC}, this Betti moduli space is homeomorphic to the stable
moduli space of parabolic Higgs bundles with the corresponding Dolbeault-side
residue diagram $P^{\vec\alpha,\vec\xi}_{\vec m}$:
\[
\mathcal M^{s}_{\mathrm{Dol}}(P^{\vec\alpha,\vec\xi}_{\vec m})
\simeq_{\mathrm{top}}
\mathcal M^{s}_{\mathrm{Betti}}(P^{\vec\beta,\vec\lambda}_{\vec m}).
\]
In the case where the Dolbeault residue acts by a scalar on each parabolic
graded piece, the Dolbeault-side moduli space is precisely the moduli space of
$\vec\xi$-parabolic Higgs bundles studied above. Therefore, the connectedness of $\Higgs^{\vxi\text{-par}}(C,D,\vm,\va)$ (Theorem~\ref{thm:xipar})
induces the following connectedness result on the Betti side. 

\begin{theorem}[Theorem \ref{thm:connected-Betti-xi-par}]\label{thm:intro-Betti-filtered-connected}
	Let $P^{\vec\beta,\vec\lambda}_{\vec m}$ be a Betti-side residue diagram whose
	corresponding Dolbeault-side residue diagram, via Simpson's table \ref{table:simpson}, is of scalar graded
	residue data of type $(\vec m,\vec\xi)$. Suppose that a generic system of parabolic weights exists as in Assumption~\ref{assmp:genericity}. If $\mathcal M^{s}_{\mathrm{Betti}}(P^{\vec\beta,\vec\lambda}_{\vec m})$
	is nonempty, then it is connected.
\end{theorem}

A particularly important special case occurs when the Betti filtration is
trivial, i.e. when all filtration jumps are $\beta=0$. Then a filtered local
system is just an ordinary local system on $C^\circ$ with prescribed local
monodromy, which is known as a character variety. 

Let $\mathcal C=(\mathcal C_1,\ldots,\mathcal C_n)$ be a collection of conjugacy
classes in $GL_r(\C)$ satisfying the determinant condition
\[
\prod_{i=1}^n\det(\mathcal C_i)=1.
\]
Assume moreover that, for each $i$, all eigenvalues of $\mathcal C_i$ have the
same argument. Under Simpson's table \ref{table:simpson}, this is exactly the condition that the
compatible Dolbeault-side parabolic weights are constant at each puncture, say
$\underline{\alpha}_i=(\alpha_{i,1})$. We write $\chi(C^\circ,\mathcal C)$ for the
character variety of representations of $\pi_1(C^\circ)$ whose local monodromy
around $p_i$ lies in $\mathcal C_i$, and $\chi(C^\circ,\mathcal C)^{\mathrm{irr}}$
for its irreducible locus. Combining Theorem~\ref{thm:poisson-closure} with Theorem \ref{thm:tame-NAHC}, we obtain the following connectedness statement.

\begin{theorem}[Theorem \ref{thm:connected-character-variety}]\label{thm:intro-character-connected}
Let $\mathcal C=(\mathcal C_1,\ldots,\mathcal C_n)$ be $q$-indivisible such that $\prod_{i=1}^n\det(\mathcal C_i)=1.$ Assume that, for each $i$, all the eigenvalues of $\mathcal C_i$ have a common argument. Then the irreducible locus
	$\chi(C^\circ,\mathcal C)^{\operatorname{irr}}$
	is connected whenever it is nonempty. 
\end{theorem}

Connectedness of $\chi(C^\circ,\mathcal C)$ for generic semisimple
$\mathcal C$, and for curves of arbitrary genus, was established by
Hausel--Letellier--Rodriguez-Villegas \cite{HauselRodriguez2013}.
Theorem~\ref{thm:intro-character-connected} extends this picture in some
directions by allowing the prescribed monodromy classes $\mathcal C_i$ to be
nongeneric, under the common-argument assumption on the eigenvalues. In genus
zero, related connectedness results for certain nongeneric local monodromy data
were obtained by Shu \cite{Shu2025DSP} via multiplicative quiver varieties.

	Finally, it is natural to ask whether the symplectic resolution constructed on the
	Dolbeault side for the normalization of $\Higgs^{\overline{\cO},ss}(C,D)$ has a
	Betti-side counterpart under the tame non-abelian Hodge correspondence. A Springer-type resolution of the character varieties with fixed $GL_r(\C)$-conjugacy classes at the punctures, obtained by incorporating flags, already appears in the work of Letellier \cite{Letellier2013CharacterClosures}. This question is related to Simpson's isosingularity philosophy \cite{simpsonmoduli, Tirelli2019SymplecticResolutionsHiggs} and
	to the Betti-side symplectic resolutions studied by Schedler--Tirelli \cite{SchedlerTirelli2022}.

\section{Poisson geometry}
In this article, we work over the field of complex numbers $\C$. We recall some standard facts on holomorphic Poisson structures. Let $X$ be a smooth complex variety.

\begin{definition}
A holomorphic Poisson structure on $X$ is a bivector
\[
\pi\in H^0(X,\wedge^2T_X)
\]
such that $[\pi,\pi]=0$, where $[-,-]$ denotes the Nijenhuis--Schouten bracket. Equivalently, $\pi$ defines a bracket $\{-,-\}$ on $\mathcal O_X$ by $\{f,g\}:=\pi(df,dg)$, and this bracket is skew-symmetric, satisfies the Leibniz rule, and satisfies the Jacobi identity. A smooth variety equipped with a holomorphic Poisson structure is called a holomorphic Poisson variety.
\end{definition}

The bivector $\pi$ induces an $\mathcal O_X$-linear map
\[
\pi^\sharp:T_X^*\longrightarrow T_X,
\qquad
\alpha\longmapsto \pi(\alpha,-).
\]
The image of $\pi^\sharp$ is called the characteristic distribution of the Poisson structure.

A holomorphic symplectic variety is the nondegenerate case of a Poisson variety. Indeed, suppose that $\pi^\sharp$ is an isomorphism. Let $\omega^\flat:T_X\to T_X^*$ be its inverse, and let $\omega\in H^0(X,\Omega_X^2)$ be the corresponding nondegenerate skew-symmetric two-form. The closedness of $\omega$ is equivalent to the Poisson condition for $\pi$:
\[
d\omega=0
\qquad\Longleftrightarrow\qquad
[\pi,\pi]=0.
\]
Indeed, let $\alpha,\beta,\gamma$ be local one-forms, and set $X=\pi^\sharp(\alpha)$, $Y=\pi^\sharp(\beta)$, and $Z=\pi^\sharp(\gamma)$. With the standard convention for the Nijenhuis--Schouten bracket, one has
\[
\frac{1}{2}[\pi,\pi](\alpha,\beta,\gamma)
=
d\omega(X,Y,Z),
\]
up to sign convention. Hence, $\omega\in H^0(X,\Omega_X^2)$ is a nondegenerate closed holomorphic two-form and this is called a holomorphic symplectic form.

\begin{definition}
A symplectic leaf of a holomorphic Poisson variety $(X,\pi)$ is a maximal connected locally closed smooth subvariety tangent to the characteristic distribution. Equivalently, it is a connected subvariety $L\subset X$ such that $T_xL=\operatorname{Im}(\pi_x^\sharp)$
for every $x\in L$, and which is maximal with this property.
\end{definition}

On each symplectic leaf $L$, the restriction of $\pi$ is nondegenerate as a bivector on $L$. Its inverse is a holomorphic symplectic form on $L$. Thus a Poisson variety is foliated by symplectic leaves, whose dimensions may vary according to the rank of $\pi$.

We will use the following criterion for restricting a Poisson bivector to a smooth subvariety.

\begin{definition}
Let $(X,\pi)$ be a holomorphic Poisson variety and let $Y\subset X$ be a smooth subvariety. We say that $\pi$ is tangent to $Y$ if, for every $y\in Y$, the image of
\[
\pi_y^\sharp:T_y^*X\to T_yX
\]
is contained in $T_yY$. Equivalently, the restriction $\pi|_Y$, originally a section of $\wedge^2T_X|_Y$, actually lies in $H^0(Y,\wedge^2T_Y)$.
\end{definition}

\begin{prop}\label{prop:symp leaf closed}
Let $(X,\pi)$ be a holomorphic Poisson variety, and let $Y\subset X$ be a smooth subvariety. Suppose that along $Y$ the map $\pi^\sharp$ factors as
\[
T_X^*|_Y
\longrightarrow
T_Y^*
\xrightarrow{\ \sigma\ }
T_Y
\longrightarrow
T_X|_Y,
\]
where the first map is the natural restriction, the last map is induced by the inclusion $Y\subset X$, and $\sigma:T_Y^*\to T_Y$ is an isomorphism. Then the two-form $\omega_Y$ induced by the inverse $\sigma^{-1}$ is closed. Hence $(Y,\omega_Y)$ is a holomorphic symplectic variety.
\end{prop}

\begin{proof}
The factorization implies that $\pi$ is tangent to $Y$. Therefore $\pi$ restricts to a bivector $\pi_Y\in H^0(Y,\wedge^2T_Y)$, and the induced map $\pi_Y^\sharp:T_Y^*\to T_Y$ is precisely $\sigma$. Hence $\pi_Y$ is nondegenerate and we have the induced holomorphic two form. 

Since $\pi$ is Poisson, $[\pi,\pi]=0$. The Nijenhuis--Schouten bracket is compatible with restriction of tangent bivectors to a smooth subvariety, so $[\pi_Y,\pi_Y]=0$. By the equivalence above for nondegenerate bivectors, this is equivalent to $d\omega_Y=0$.
\end{proof}

\begin{example}\label{ex:GLn-linear algebra}
Let $G$ be a complex reductive group with Lie algebra $\mathfrak g$. The dual vector space $\mathfrak g^*$ carries the Lie--Poisson structure, whose symplectic leaves are the coadjoint orbits \cite[Section 1.3]{ChrissGinzburg1997}.

For $G=GL_r(\C)$, we identify its Lie algebra $\mathfrak{gl}_r$ with $\mathfrak{gl}_r^*$ by the nondegenerate trace pairing
\[
\mathfrak{gl}_r\times \mathfrak{gl}_r\longrightarrow \mathbb C,
\qquad
(A,B)\longmapsto \operatorname{tr}(AB).
\]
 Under this identification, the coadjoint action of $GL_r$ on $\mathfrak{gl}_r^*$ becomes the usual conjugation action on $\mathfrak{gl}_r$. Hence the coadjoint orbits are precisely the conjugacy classes in $\mathfrak{gl}_r$.

If $\mathcal O\subset\mathfrak{gl}_r$ is the conjugacy class of $A$, then its tangent space at $A$ is $T_A\mathcal O=[A,\mathfrak{gl}_r]$. Indeed, differentiating the conjugation action $g\cdot A=gAg^{-1}$ at the identity gives $X\mapsto [X,A]$, which has the same image as $X\mapsto [A,X]$ up to sign. We write $\omega^{\cO}$ for the induced symplectic form on $\cO$.  If tangent vectors are written as $[X,A]$ and $[Y,A]$, then $\omega_A^{\cO}([X,A],[Y,A])
=
\operatorname{tr}(A[X,Y])$. 
For example, the trace-orthogonal complement of $T_A\mathcal O=[A,\mathfrak{gl}_r]$ is the centralizer $Z_{\mathfrak{gl}_r}(A)$, since $\operatorname{tr}(B[A,X])=
\operatorname{tr}([B,A]X)$. 
Thus $B$ annihilates $T_A\mathcal O$ under the trace pairing if and only if $[B,A]=0$.
\end{example}

\section{Meromorphic Higgs bundles}\label{sec:stacky-defs}

We now introduce one of the main objects of this paper: the moduli space of meromorphic Higgs bundles with fixed residue conjugacy classes. For completeness, we begin in Section~\ref{sec: construction mero moduli} by giving algebraic definitions of the relevant moduli spaces. We then discuss the Poisson geometry of these moduli spaces by studying infinitesimal deformations in Section~\ref{sec:mero deform}. 

\subsection{Nitsure's construction and residue conditions}\label{sec: construction mero moduli}

We recall the GIT construction of the moduli space of $K_C(D)$-twisted Higgs bundles following Nitsure \cite{Nitsure1991curve}.

For $N\gg 0$, every semistable $K_C(D)$-twisted Higgs bundle $(E,\Phi)$ of rank $r$ and degree $d$ satisfies $H^1(C,E(N))=0$, and $E(N)$ is globally generated. Hence
$h^0(C,E(N))=\chi(E(N))=:P(N)$
is independent of $(E,\Phi)$. Fix a vector space $V$ of dimension $P(N)$.

Let $Q$ be the relevant open subscheme of the Quot scheme parametrizing quotients
$q:V\otimes\cO_C(-N)\twoheadrightarrow E$
such that $E$ has rank $r$ and degree $d$, and such that the induced map
$V\xrightarrow{\sim}H^0(C,E(N))$ is an isomorphism. Let
$V\otimes\cO_{C\times Q}(-N)\twoheadrightarrow\mathcal E$
be the universal quotient. Adding the Higgs field gives a parameter scheme $R$ over $Q$ whose points consist of a quotient
$q:V\otimes\cO_C(-N)\twoheadrightarrow E$
together with a Higgs field
$\Phi:E\to E\otimes K_C(D)$.
The group $G:=\GL(V)$ acts on $R$ by changing the framing.  Since scalar matrices act trivially on the corresponding Higgs bundle, one may equivalently work with $\mathrm{PGL}(V)$.

Nitsure constructs a $G$-linearization such that the GIT-semistable points are precisely the slope-semistable Higgs bundles. Let $R^{ss}\subset R$ and $R^s\subset R^{ss}$ denote the semistable and stable loci. The coarse moduli space of semistable meromorphic Higgs bundles is
\[
\Higgs^{ss}(C,D):=R^{ss}//G.
\]
Its closed points correspond to S-equivalence classes of semistable Higgs bundles. The stable locus is the geometric quotient
\[
\Higgs^s(C,D):=R^s/G\subset\Higgs^{ss}(C,D),
\]
and its points correspond to isomorphism classes of stable meromorphic Higgs bundles. 

\begin{rem}As $\Higgs^{ss}(C,D)$ and $\Higgs^s(C,D)$ may not be connected, we follow the convention of \cite{Bottacin1995symplectic,markmanspectral} and restrict our focus to the connected component of the semistable locus containing $(E, \Phi)$ where the underlying $E$ is stable. By an abuse of notation, we will retain the notation $\Higgs^{ss}(C, D)$ and $\Higgs^{s}(C, D)$ for this component for the remainder of this paper. 
\end{rem}

We now impose residue conditions. Let $(\mathcal{E}, \Phi_{\mathrm{univ}})$ be a universal family of Higgs bundles on $C\times R$. A universal Higgs field $\Phi_{\mathrm{univ}}$ on $R$ has residues along the marked points. For each $p_i\in D$, restriction to $\{p_i\}\times R$ gives a relative endomorphism
\[
\res_{p_i}(\Phi_{\mathrm{univ}})\in \End(\mathcal E_{p_i}).
\]
where $\mathcal{E}_{p_i}$ is the restriction of $\mathcal{E}$ to $\{p_i\} \times R$. 
Equivalently, the residues define a $G$-equivariant section of the relative endomorphism bundle
\[
\res_D:R\longrightarrow \bigoplus_{i=1}^n\End(\mathcal E_{p_i}),
\]
where $G$ acts by conjugation on the fibers. On the other hand, let $P_i:=\Fr(\mathcal E_{p_i})\to R$
be the frame bundle of $\mathcal E_{p_i}$. Thus $P_i$ is a principal
$\GL_r$-bundle, and we have a natural identification $\End(\mathcal E_{p_i})
\cong
P_i\times^{\GL_r}\mathfrak{gl}_r$,
where $\GL_r$ acts on $\mathfrak{gl}_r$ by conjugation. 

Now let $\cO=\cO_1\times\cdots\times\cO_n$ be a product of adjoint orbits, and let
$\ov\cO=\ov\cO_1\times\cdots\times\ov\cO_n$
be the corresponding product of orbit closures satisfying the condition \eqref{eq:trace-compatibility}. Then we set 
\[
\begin{aligned}
    {\cO}_i(\mathcal E_{p_i}) :=
P_i\times^{\GL_r}{\cO}_i
\subset
\End(\mathcal E_{p_i}),\qquad 
\ov{\cO}_i(\mathcal E_{p_i}) :=
P_i\times^{\GL_r}\ov{\cO}_i
\subset
\End(\mathcal E_{p_i}). 
\end{aligned}
\]
\noindent Define
\[
\begin{aligned}
R^{\cO,ss}:=R^{ss}\cap\res_D^{-1}\left(\prod_{i=1}^n\cO_i(\mathcal{E}_{p_i})\right), &
\qquad
R^{\cO,s}:=R^s\cap\res_D^{-1}\left(\prod_{i=1}^n\cO_i(\mathcal{E}_{p_i})\right),\\
R^{\ov\cO,ss}:=R^{ss}\cap\res_D^{-1}\left(\prod_{i=1}^n\ov{\cO}_i(\mathcal{E}_{p_i})\right), &
\qquad
R^{\ov\cO,s}:=R^s\cap\res_D^{-1}\left(\prod_{i=1}^n\ov{\cO}_i(\mathcal{E}_{p_i})\right).
\end{aligned}
\]
\noindent where $\prod_{i=1}^n \cO_i(\mathcal E_{p_i})$ is regarded as a relative
locally closed subset of $\bigoplus_{i=1}^n \End(\mathcal E_{p_i})$ over $R$. Since $\cO$ and $\ov\cO$ are invariant under conjugation, these are $G$-invariant subschemes. We define
\[
\begin{aligned}
    \Higgs^{\cO,ss}(C,D):=R^{\cO,ss}//G, &
\qquad
\Higgs^{\cO,s}(C,D):=R^{\cO,s}/G, \\
\Higgs^{\ov\cO,ss}(C,D):=R^{\ov\cO,ss}//G, &
\qquad
\Higgs^{\ov\cO,s}(C,D):=R^{\ov\cO,s}/G.
\end{aligned}
\]

We record the basic topological properties of these moduli spaces.

\begin{lemma}\label{lem:residue-loci-closed-locally-closed}
Let the notation be as above. 
\begin{enumerate}
    \item $\Higgs^{\ov\cO,ss}(C,D)\subset\Higgs^{ss}(C,D)$ and $\Higgs^{\ov\cO,s}(C,D)\subset\Higgs^s(C,D)$ are closed.
    \item $\Higgs^{\ov\cO,s}(C,D)\subset\Higgs^{\ov\cO,ss}(C,D)$ and $\Higgs^{\cO,s}(C,D)\subset\Higgs^{\cO,ss}(C,D)$ are open.
    \item $\Higgs^{\cO,s}(C,D)\subset\Higgs^{\ov\cO,s}(C,D)$ is open.
\end{enumerate}
\end{lemma}
\begin{proof}
	Since $\ov\cO$ is closed in $\prod_{i=1}^n \mathfrak{gl}_r$, the space  $R^{\ov\cO,ss}$ is a closed $G$-invariant subscheme of
	$R^{ss}$. Since $\pi:R^{ss}\to \Higgs^{ss}(C,D)$ is a good quotient, the image
	of a closed $G$-invariant subset is closed. Thus
	$\Higgs^{\ov\cO,ss}(C,D)$ is closed in $\Higgs^{ss}(C,D)$.

	Intersecting with the open stable locus gives
	$\Higgs^{\ov\cO,s}(C,D)
	=\Higgs^{\ov\cO,ss}(C,D)\cap \Higgs^s(C,D)$.
	Therefore $\Higgs^{\ov\cO,s}(C,D)$ is closed in $\Higgs^s(C,D)$ and open in
	$\Higgs^{\ov\cO,ss}(C,D)$. This proves (1) and the first statement of (2).

	The same intersection argument gives
	$\Higgs^{\cO,s}(C,D)
	=\Higgs^{\cO,ss}(C,D)\cap \Higgs^s(C,D)$,
	so $\Higgs^{\cO,s}(C,D)$ is open in $\Higgs^{\cO,ss}(C,D)$. This proves the
	second statement of (2).

	Finally, each orbit $\cO_i$ is open in its closure $\ov\cO_i$. Hence
	$\cO=\prod_i\cO_i$ is open in $\ov\cO=\prod_i\ov\cO_i$. Therefore
	$R^{\cO,s}$ is open in $R^{\ov\cO,s}$. Since the stable quotient is geometric,
	this descends to an open inclusion
	$\Higgs^{\cO,s}(C,D)\subset \Higgs^{\ov\cO,s}(C,D)$.
	This proves (3).
\end{proof}

\begin{remark}\label{rem:semistable-residue-condition}
	Although $R^{\cO,ss}$ is open in $R^{\ov\cO,ss}$, it is generally not saturated
	with respect to the good quotient map
	$R^{\ov\cO,ss}\to R^{\ov\cO,ss}//G$. Indeed, passing to an $S$-equivalent
	polystable representative may move the residue from $\cO$ to the boundary
	$\ov\cO\setminus\cO$. Thus we do not claim that $\Higgs^{\cO,ss}(C,D)$ is
	open in $\Higgs^{\ov\cO,ss}(C,D)$.
\end{remark}

There is a well-defined Hitchin map \[h:\Higgs^{ss}(C,D) \to A:=\bigoplus_{j=1}^rH^0(C,K_C(D)^{\otimes j})\]
defined by the coefficients of the characteristic polynomial of the Higgs field $\Phi$. We will revisit how we describe the base of the restriction over $\Higgs^{\ov\cO,ss}$ in Section~\ref{sec:integrable systems}. 

We next explain why the orbit-closure loci are singular in general. We work on the open locus considered by Markman \cite{markmanspectral}, where the underlying vector bundle is stable and a level structure is chosen along $D$.

Let $\mathcal U_D$ denote the moduli space of stable vector bundles equipped with a $D$-level structure. A point of $\mathcal U_D$ is a pair $(E,\eta)$, where $E$ is a stable vector bundle of rank $r$ and degree $d$, and $\eta:E|_D\xrightarrow{\sim}V_D$ is a framing along $D$. By Markman's description, $T^*\mathcal U_D$ parametrizes triples $(E,\Phi,\eta)$, where $(E,\eta)\in\mathcal U_D$ and
$\Phi\in H^0(C,\End(E)\otimes K_C(D))$.

Using the level structure, the residues of $\Phi$ become matrices. Thus we obtain the residue map $\mu:T^*\mathcal U_D\longrightarrow \mathfrak R$,
where
\[
\mathfrak R:=
\left\{
(A_i)_i\in\bigoplus_i\mathfrak{gl}_r
\ \middle|\
\sum_i\operatorname{tr}(A_i)=0
\right\}.
\]
Here the subscript $(-)_i$ denotes the $i$-th component of the direct sum. Explicitly,
$\mu(E,\Phi,\eta)=
\bigl(\eta_i\circ \res_{p_i}(\Phi)\circ\eta_i^{-1}\bigr)_i.$

\begin{lemma}\label{lem:markman-residue-submersion}
	The residue map $\mu:T^*\mathcal U_D\to\mathfrak R$ is smooth.
\end{lemma}

\begin{proof}
	It is enough to show that $d\mu$ is surjective at every point. Let
	$y=(E,\Phi,\eta)$. Since $E$ is stable, $H^0(C,\End(E))=\C\cdot\id_E$. Consider the residue map for Higgs-field variations
	\[
	\res_D:
	H^0(C,\End(E)\otimes K_C(D))
	\longrightarrow
	\bigoplus_i\End(E_{p_i}).
	\]
	Its image is contained in the  residue space
	\[
	\mathfrak R_E:=
	\left\{
	(B_i)_i\in\bigoplus_i\End(E_{p_i})
	\ \middle|\
	\sum_i\tr(B_i)=0
	\right\}.
	\]
	We claim that $\res_D$ is surjective onto $\mathfrak R_E$. Indeed, the residue sequence
	\[
	0\to\End(E)\otimes K_C
	\to\End(E)\otimes K_C(D)
	\to \bigoplus_i\End(E_{p_i}) \to 0
	\]
	gives a connecting homomorphism $\delta:\bigoplus_i\End(E_{p_i})
	\longrightarrow H^1(C,\End(E)\otimes K_C)$,
	and $\operatorname{coker}(\res_D)$ identifies with $\operatorname{Im}(\delta)$.
	By Serre duality and the trace pairing, the dual of $\delta$ is the evaluation
	map
	\[
	H^0(C,\End(E))
	\longrightarrow
	\left(\bigoplus_i\End(E_{p_i})\right)^*,
	\qquad
	s\longmapsto \bigl((B_i)_i\mapsto \sum_i\tr(s(p_i)B_i)\bigr).
	\]
    where $(-)^*$ denotes the linear dual. By the assumption that $E$ is stable, $H^0(C,\End(E))=\C\cdot\id_E$, hence the annihilator of
	$\operatorname{Im}(\res_D)=\ker(\delta)$ is spanned by the functional $(B_i)_i\longmapsto \sum_i\tr(B_i)$. Thus
	\[
	\operatorname{Im}(\res_D)
	=
	\left\{
	(B_i)_i\in\bigoplus_i\End(E_{p_i})
	\ \middle|\
	\sum_i\tr(B_i)=0
	\right\}
	=\mathfrak R_E.
	\]

	Now vary only the Higgs field, keeping $E$ and the framing $\eta$ fixed:
	$\Phi_\varepsilon=\Phi+\varepsilon\psi$, where
	$\psi\in H^0(C,\End(E)\otimes K_C(D))$. Then $
	d\mu_y(\psi)
	=
	\bigl(\eta_i\circ\res_{p_i}(\psi)\circ\eta_i^{-1}\bigr)_i$. Conjugation by the framings $\eta_i$ identifies $\mathfrak R_E$ with
	$\mathfrak R$. Since $\res_D$ is surjective onto $\mathfrak R_E$, the image of
	$d\mu_y$ contains all of $\mathfrak R$. Hence $d\mu_y$ is surjective.
\end{proof}
As before, take $\cO$ and its closure $\ov\cO$ satisfying the residue condition \eqref{eq:trace-compatibility}. Note that the singular locus of $\ov\cO$, denoted by $\mathrm{Sing}(\ov\cO)$, is given by $\ov\cO \setminus \cO$ (see \cite{kaledinsingularity, namikawa2005birationalgeometrysymplecticresolutions, FuSingularity} and references therein). Define
$(T^*\mathcal U_D)^{\ov\cO}:=\mu^{-1}(\ov\cO)$.

\begin{proposition}\label{prop:singularity-criterion}
	Let $y=(E,\Phi,\eta)\in (T^*\mathcal U_D)^{\ov\cO}$, and write $A=\mu(y)$. If $A\in\operatorname{Sing}(\ov\cO)$, then $(T^*\mathcal U_D)^{\ov\cO}$ is singular at $y$.
\end{proposition}

\begin{proof}
	By Lemma~\ref{lem:markman-residue-submersion}, $\mu$ is smooth at $y$. Since
	$(T^*\mathcal U_D)^{\ov\cO}=\mu^{-1}(\ov\cO)$, the formal local structure theorem for smooth morphisms gives a noncanonical isomorphism
	\[
	\widehat{\cO}_{(T^*\mathcal U_D)^{\ov\cO},y}
	\simeq
	\widehat{\cO}_{\ov\cO,A}[[t_1,\ldots,t_N]]
	\]
	for some $N$. If $A\in\operatorname{Sing}(\ov\cO)$, then $\widehat{\cO}_{\ov\cO,A}$ is not regular. A formal power series ring over a nonregular complete local ring is again nonregular. Hence $(T^*\mathcal U_D)^{\ov\cO}$ is singular at $y$.
\end{proof}

We now pass from $T^*\mathcal{U}_D$ to the Higgs moduli. Let $G_D$ be the level group acting on the framing $\eta$. The group $G_D$ acts on $T^*\mathcal U_D$, and $\mu$ is equivariant for the adjoint action. Since $\ov\cO$ is invariant under conjugation, the locus $(T^*\mathcal U_D)^{\ov\cO}$ is $G_D$-invariant.

On the Higgs-stable locus, the $G_D$-action is free (see \cite[Lemma 6.7]{markmanspectral}): an element fixing $(E,\Phi,\eta)$ gives an automorphism of the stable Higgs pair $(E,\Phi)$ preserving the level structure, hence is scalar, and the level structure kills the scalar. Therefore the quotient map is smooth, and the quotient identifies with the corresponding Markman open subset of $\Higgs^{\ov\cO,s}(C,D)$.

\begin{corollary}\label{cor:singularity-downstairs}
	Let $(E,\Phi)\in\Higgs^{\ov\cO,s}(C,D)$ be a stable Higgs bundle whose underlying vector bundle $E$ is stable. Suppose that the residue $(\res_{p_i}\Phi)_i$ belongs to $\operatorname{Sing}(\ov\cO)$. Then $\Higgs^{\ov\cO,s}(C,D)$ is singular at $(E,\Phi)$.
\end{corollary}

\begin{proof}
	Choose a level structure $\eta$ and set $y=(E,\Phi,\eta)$. Then $\mu(y)$ is conjugate to the residue tuple and hence lies in $\operatorname{Sing}(\ov\cO)$. By Proposition~\ref{prop:singularity-criterion}, $(T^*\mathcal U_D)^{\ov\cO}$ is singular at $y$. Since the quotient map from the level-structured locus to the corresponding open subset of $\Higgs^{\ov\cO,s}(C,D)$ is smooth, smoothness downstairs would imply smoothness upstairs, a contradiction.
\end{proof}

\begin{example}
	Let $r=2$ and $D=\{p\}$. Let $\mathcal N\subset\mathfrak{sl}_2$ be the nilpotent cone, i.e. the closure of the nonzero nilpotent orbit. Writing
	\[
	A=
	\begin{pmatrix}
	a&b\\
	c&-a
	\end{pmatrix},
	\]
	we have $\mathcal N=\{a^2+bc=0\}\subset\mathfrak{sl}_2$. Hence
	$\operatorname{Sing}(\mathcal N)=\{0\}$. Therefore, if $E$ is stable and $(E,\Phi)$ is a stable meromorphic Higgs bundle with $\res_p(\Phi)=0$, then $\Higgs^{\mathcal N,s}(C,p)$ is singular at $(E,\Phi)$. Formally locally, the singularity contains the factor $
	\C[[a,b,c]]/(a^2+bc)$, up to a formal smooth factor.
\end{example}

\subsection{The deformation theory}\label{sec:mero deform}

We study the infinitesimal deformation theory for $\Higgs^{\cO,s}(C,D)$ and compare it with the characteristic distribution of the Bottacin--Markman Poisson structure on $\Higgs^s(C,D)$ \cite{markmanspectral,Bottacin1995symplectic}.

Let $\mathcal E=(E,\Phi)$ be a stable meromorphic Higgs bundle. Consider the complex
\[
C_{\text{mero},\mathcal E}^\bullet:
\qquad
\End(E)\xrightarrow{[\Phi,\cdot]}\End(E)\otimes K_C(D),
\]
placed in degrees $0,1$. Then it is well known that $T_{\mathcal E}\Higgs^s(C,D)\cong\HH^1(C_{\mathrm{mero},\mathcal E}^\bullet)$. 

Let $\End(E)^\vee$ be a $\cO_C$-linear dual of $\End(E)$. The trace pairing induces an isomorphism $\End(E)^\vee \cong \End(E)$. Then the $\cO_C$-linear dual is identified with 
\[
(C_{\text{mero},\mathcal E}^\bullet)^\vee:
\qquad
\End(E)\otimes K_C^{-1}(-D)\xrightarrow{-[\Phi,\cdot]}\End(E)
\]
placed in degrees $-1,0$. We define the Serre-dual complex $D(C^\bullet_{\text{mero},\mathcal{E}}):=(C_{\text{mero},\mathcal E}^\bullet)^\vee \otimes K_C[-1]$. Then Serre duality gives $T_{\mathcal E}^*\Higgs^s(C,D)\cong\HH^1(D(C_{\text{mero},\mathcal E}^\bullet))$. In particular, there is a canonical inclusion map of complexes $D(C_{\text{mero},\mathcal E}^\bullet) \to C_{\text{mero},\mathcal E}^\bullet$ given by 

\[\begin{tikzcd}
    \End(\mathcal E)(-D)
    \arrow[r, "{-[\Phi,\cdot]}"]
    \arrow[d, hook, "id"]
    &
    \End(\mathcal E)\otimes K_C
    \arrow[d, hook, "-id \otimes \iota"]
    \\
    \End(\mathcal E)
    \arrow[r, "{[\Phi,\cdot]}"]
    &
    \End(\mathcal E)\otimes K_C(D).
\end{tikzcd}\]
where $\iota:K_C \hookrightarrow K_C(D)$ is a canonical inclusion. Then the induced morphism on the hypercohomology $\HH^1(-)$ gives so-called Bottacin--Markman Poisson map 
\[\pi_{\mathcal E}^\sharp:T_{\mathcal E}^*\Higgs^s(C,D)\to T_{\mathcal E}\Higgs^s(C,D).\]

\noindent Now let $\mathcal E=(E,\Phi)\in\Higgs^{\cO,s}(C,D)$ and set
$A_i:=\res_{p_i}(\Phi)\in\cO_i$. Define
\[
\cK_{\cO,\mathcal E}
:=
\left\{
s\in\End(E)\otimes K_C(D)
\ \middle|\
\res_{p_i}(s)\in T_{A_i}\cO_i
\text{ for all }i
\right\}.
\]
Since $T_{A_i}\cO_i=[A_i, \End(E)_{p_i}]$, we have a well-defined complex 
\[
C_{\cO,\mathcal E}^\bullet:
\qquad
\End(E)\xrightarrow{[\Phi,\cdot]}\cK_{\cO,\mathcal E},
\]
 placed in degrees $0,1$. 

\begin{proposition}\label{prop:tangent-CO}
	There is a canonical identification $T_{\mathcal E}\Higgs^{\cO,s}(C,D)\cong \HH^1(C_{\cO,\mathcal E}^\bullet)$.
\end{proposition}

\begin{proof}
	Choose an affine open cover $\mathfrak U=(U_\alpha)$ such that each marked point lies in a unique open set. A first-order deformation as a meromorphic Higgs bundle is represented by a \v Cech cocycle
	\[
	(u_{\alpha\beta},v_\alpha)\in
	\check C^1(\mathfrak U,\End(E))
	\oplus
	\check C^0(\mathfrak U,\End(E)\otimes K_C(D))
	\]
	satisfying $\delta u=0$ and $\delta v=[\Phi,u]$, where $\delta$ is a \v Cech differential. 

	The deformation preserves the residue orbit condition if and only if, for every $i$, the first-order residue variation $\res_{p_i}(v_\alpha)$ lies in $T_{A_i}\cO_i$ on the open set containing $p_i$. This is exactly the condition that $v_\alpha$ is a section of $\cK_{\cO,\mathcal E}$. Therefore first-order deformations inside $\Higgs^{\cO,s}(C,D)$ are governed by $C_{\cO,\mathcal E}^\bullet$, and the tangent space is its first hypercohomology.
\end{proof}

Next, we describe the Serre-dual complex $D(C^\bullet_{\cO,\mathcal{E}})$. Note that by definition, $\cK_{\cO, \mathcal{E}}$ is locally free and satisfies
\[ \End(E)\otimes K_C \subset \cK_{\cO, \mathcal{E}} \subset \End(E)\otimes K_C(D). \] By taking the $\cO_C$-linear dual, we have 
\[\cK^{\vee}_{\cO, \mathcal{E}}\otimes K_C\to \End(E)^\vee  \] 
which is an inclusion, because $\cK^{\vee}_{\cO, \mathcal{E}}\otimes K_C$ is torsion free, so any kernel must vanish. Since $\End(E)^\vee\cong \End(E)$ under the trace pairing, we can identify $\cK^{\vee}_{\cO, \mathcal{E}}\otimes K_C$ as a subsheaf of $\End(E)$. More precisely, the trace pairing sends a section $s$ in $\End(E)$ to a local map $\End(E)\otimes K_C(D)\to K_C(D)$ which restricts to $f:\cK_{\cO,\mathcal E}\to K_C(D)$. Then $f$ is a section of $\cK_{\cO,\mathcal E}^\vee\otimes K_C= \mathcal Hom(\cK_{\cO, \mathcal E}, K_C)$ if and only if the residues of $f(t) = \tr(s\cdot t)$ are zero for all $t\in \cK_{\cO,\mathcal E}$. Recall from Example~\ref{ex:GLn-linear algebra} that the orthogonal complement of $T_{A_i}\cO_i$ under the trace pairing is the centralizer of $A_i$, denoted by $Z_{\End(E)_{p_i}}(A_i)$. Hence, we have
\[
\mathcal{K}_{\cO,\mathcal{E}}^\vee \otimes K_C=\left\{s \in \End(E) \ \middle|\ s(p_i) \in Z_{\End(E)_{p_i}}(A_i)\text{ for all } i\right\}
\]
which will be denoted by $\mathcal{Z}_{\cO,\mathcal{E}}$. Then the Serre-dual complex  $D(C^\bullet_{\cO,\mathcal{E}})$ is given by 
\[
D(C^\bullet_{\cO,\mathcal{E}}): \qquad \mathcal{Z}_{\cO,\mathcal{E}} \xrightarrow{-[\Phi, \cdot]} \End(E)\otimes K_C
\]
placed in degrees $0,1$. 

\begin{proposition}\label{prop:smooth-HiggsO}
	The moduli space $\Higgs^{\cO,s}(C,D)$ is smooth of dimension
	\[
	\dim \Higgs^{\cO,s}(C,D)
	=
	2r^2(g-1)+2+\sum_{i=1}^n\dim\cO_i.
	\]
\end{proposition}

\begin{proof}
Since $\mathcal E$ is stable, we have $\HH^0(C_{\cO,\mathcal E}^\bullet)=\C\cdot\id_E$. By Serre duality, the obstruction space is given by $
\HH^2(C_{\cO,\mathcal E}^\bullet)
\cong
\HH^0(D_{\cO,\mathcal E}^\bullet)^\vee$. A section
contributing to $\HH^0(D_{\cO,\mathcal E}^\bullet)$ is a global endomorphism of
$E$ commuting with $\Phi$, with the required centralizer condition at the marked
points. By stability, such an endomorphism is scalar. Conversely, scalar
endomorphisms clearly define elements of $\HH^0(D_{\cO,\mathcal E}^\bullet)$.
Therefore $\HH^2(C_{\cO,\mathcal E}^\bullet)=\C \cdot \mathrm{id}_E$. Thus the trace-free obstruction space vanishes, and $\Higgs^{\cO,s}(C,D)$ is
smooth. By applying the standard argument of taking the determinant of the Higgs bundles and the traceless part of the complex (see e.g. \cite[Corollary 2.11]{LeeLee2024}), it follows that the obstruction space for the traceless part of the complex vanishes, and hence, $\Higgs^{\cO,s}(C,D)$ is
smooth.

It remains to compute the dimension. The quotient of
$\End(E)\otimes K_C(D)$ by the degree-one term of
$C_{\cO,\mathcal E}^\bullet$ is supported at $D$, and its fiber at $p_i$ is $
\End(E_{p_i})/T_{\res_{p_i}(\Phi)}\cO_i$. Hence its length is $\sum_{i=1}^n(r^2-\dim\cO_i)$.
Therefore
\[
\begin{aligned}
    \chi(C_{\cO,\mathcal E}^\bullet)&=
\chi(\End(E))
-
\chi(\End(E)\otimes K_C(D))
+
\sum_{i=1}^n(r^2-\dim\cO_i)\\
&=r^2(1-g)-r^2(g-1+n)+\sum_{i=1}^n(r^2-\dim \cO_i) \\
&=2r^2(1-g)-\sum_{i=1}^n\dim \cO_i
\end{aligned}
\]
where the second equality follows from Riemann--Roch theorem. Now the result follows from Proposition~\ref{prop:tangent-CO} and the identity $\dim \HH^1(C^\bullet_{\cO, \mathcal{E}})=2-\chi(C^\bullet_{\cO,\mathcal{E}})$. 
\end{proof}

Next, we describe a symplectic form on $\Higgs^{\cO,s}(C,D)$. By Serre duality and Proposition~\ref{prop:tangent-CO}, we have $T_{\mathcal E}^*\Higgs^{\cO,s}(C,D)\cong\HH^1(D(C_{\cO,\mathcal E}^\bullet))$. In particular, there is a canonical inclusion $D(C^\bullet_{\cO, \mathcal{E}}) \hookrightarrow C^\bullet_{\cO, \mathcal{E}}$ given by 
\begin{equation}\label{e:O serredual map}
\begin{tikzcd}
    \mathcal{Z}_{\cO, \mathcal{E}}
    \arrow[r, "{-[\Phi,\cdot]}"]
    \arrow[d, hook, "inc"]
    &
    \End(\mathcal E)\otimes K_C
    \arrow[d, hook, "-id \otimes \iota"]
    \\
    \End(\mathcal E)
    \arrow[r, "{[\Phi,\cdot]}"]
    &
    \mathcal{K}_{\cO,\mathcal{E}}.
\end{tikzcd}
\end{equation}
where both vertical maps are canonical inclusions. Indeed, any section $s \in \End(E)\otimes K_C$ maps to $\mathcal{K}_{\cO, \mathcal{E}}$ because the residue of $s$ at $p_i$ is zero. Note that this inclusion is an isomorphism away from $D$ while at $D$, the cokernel is given by 
\[
\bigoplus_i \End(E)_{p_i}/Z_{\End(E)_{p_i}}(A_i) \xrightarrow{\bigoplus_i[A_i, \cdot]} \bigoplus_i T_{A_i}\cO_i
\]
As explained in Example \ref{ex:GLn-linear algebra}, the differential $\bigoplus_i[A_i, \cdot]$ becomes an isomorphism so that the hypercohomology of this cokernel complex is trivial. Therefore, the induced morphism on the hypercohomology $\HH^1(-)$ gives an isomorphism $\pi^\sharp_{\cO,\mathcal{E}}:T_{\mathcal E}^*\Higgs^{\cO,s}(C,D)\xrightarrow{\cong} T_{\mathcal E}\Higgs^{\cO,s}(C,D)$. Equivalently, we have a nondegenerate bilinear form at each tangent space
\[
\omega^{\cO}_{\mathcal E}:
T_{\mathcal E}\Higgs^{\cO,s}(C,D)\times
T_{\mathcal E}\Higgs^{\cO,s}(C,D)
\longrightarrow \C.
\]

We show that the induced $\omega^\cO$ is a symplectic form by comparing it with the Bottacin--Markman Poisson structure. Consider the following commutative diagram of complexes
\[
\begin{tikzcd}
D(C_{\text{mero},\mathcal{E}}^\bullet) \arrow[r,hook, "\mathsf{q}_1"]\arrow[d,hook] & D(C_{\cO,\mathcal{E}}^\bullet) \arrow[d,hook] \\
C_{\text{mero},\mathcal{E}}^\bullet & C_{\cO,\mathcal{E}}^\bullet \arrow[l, hook', "\mathsf{q}_2"]
\end{tikzcd}
\]
where the morphisms are canonical inclusions. By taking hypercohomology $\HH^1(-)$, it induces the following commutative diagram 
\[
\begin{tikzcd}
T^*_\mathcal{E}\Higgs^{s}(C,D) \arrow[r,"\HH^1(\mathsf{q}_1)"] \arrow[d, "\pi^\sharp_{\mathcal{E}}"] & T^*_\mathcal{E}\Higgs^{\cO,s}(C,D) \arrow[d, "\pi^\sharp_{\cO,\mathcal{E}}"] \\
T_\mathcal{E}\Higgs^{s}(C,D) & T_\mathcal{E}\Higgs^{\cO,s}(C,D) \arrow[l,"\HH^1(\mathsf{q}_2)" ]
\end{tikzcd}
\]
From the geometric viewpoint, it is clear that $\HH^1(\mathsf{q}_1)$ is surjective and $\HH^1(\mathsf{q}_2)$ is injective. One can see this directly from the deformation complex computation. For completeness, we include the proof of this fact. 

\begin{lem}\label{lem:inj,surj,co}
    Let the notation be as above. Then the following holds. 
    \begin{enumerate}
        \item $\HH^1(\mathsf{q}_1)$ is surjective;
        \item $\HH^1(\mathsf{q}_2)$ is injective. 
    \end{enumerate}
\end{lem}
\begin{proof}
    For $(1)$, consider the short exact sequence of complexes
    \[
    0 \to D(C_{\text{mero}, \mathcal{E}}^\bullet) \xrightarrow{\mathsf{q_1}}  D(C_{\cO, \mathcal{E}}^\bullet) \to \Coker(\mathsf{q_1}) \to 0
    \]
    where $\Coker(\mathsf{q_1}):\bigoplus_i Z_{\End(E)_{p_i}}(A_i)\to 0$ placed in degrees $0,1$. Therefore $\mathbb{H}^1(\Coker(\mathsf{q}_1))=0$. By taking the induced long exact sequence, the claim follows. 

    For $(2)$, consider the short exact sequence of complexes
    \[
    0 \to C_{\cO, \mathcal{E}}^\bullet \xrightarrow{\mathsf{q_2}}  C_{\text{mero}, \mathcal{E}}^\bullet \to \Coker(\mathsf{q_2}) \to 0
    \]
    where $\Coker(\mathsf{q_2}): 0\to  \bigoplus_i \End(E)_{p_i}/T_{A_i}\cO_i$. Therefore, $\mathbb{H}^0(\Coker(\mathsf{q}_2))=0$. By taking the induced long exact sequence, the claim follows. 
\end{proof}
Combining this with the standard fact about Poisson geometry recalled in
 Proposition~\ref{prop:symp leaf closed}, we obtain the following corollary. 
\begin{corollary}\label{cor:characteristic-distribution}
For $\mathcal{E}\in \Higgs^{\cO,s}(C,D)$, the image of the Bottacin--Markman Poisson map $\pi^\sharp_{\mathcal{E}}$ is precisely $T_{\mathcal E}\Higgs^{\mathcal O,s}(C,D)$. Moreover, the induced nondegenerate bivector on $\Higgs^{\mathcal O,s}(C,D)$
defines a holomorphic symplectic form $\omega^{\mathcal O}$.
\end{corollary}

\begin{theorem}\label{thm:connected HiggsO}
	If $\Higgs^{\cO,s}(C,D)\neq\emptyset$, then $\Higgs^{\cO,s}(C,D)$ is connected. Moreover, $\Higgs^{\ov\cO,s}(C,D)$ is irreducible and it becomes the closure of $\Higgs^{\cO,s}(C,D)$ in $\Higgs^s(C,D)$, denoted by $\overline{\Higgs^{\cO,s}(C,D)}$. In other words, 
	\[
	\overline{\Higgs^{\cO,s}(C,D)}
	=
	\Higgs^{\ov\cO,s}(C,D).
	\]
\end{theorem}

We postpone the proof of Theorem~\ref{thm:connected HiggsO} to the next subsection, since it uses the forgetful morphism from the corresponding $\vec\xi$-parabolic moduli space.

\begin{theorem}\label{thm:orbit-stratum-is-leaf}
	The moduli space $\Higgs^{\cO,s}(C,D)$ is a symplectic leaf of the Bottacin--Markman Poisson structure on $\Higgs^s(C,D)$.
\end{theorem}
\begin{proof}
	Let $L$ be the symplectic leaf through a point of $\Higgs^{\cO,s}(C,D)$. By Corollary~\ref{cor:characteristic-distribution}, the tangent space of $\Higgs^{\cO,s}(C,D)$ agrees with the characteristic distribution at every point. Hence $\Higgs^{\cO,s}(C,D)$ is tangent to the characteristic distribution. Since it is connected by Theorem~\ref{thm:connected HiggsO}, maximality of the symplectic leaf gives $\Higgs^{\cO,s}(C,D)\subset L$. Moreover, this inclusion is open, because both sides have the same tangent space along $\Higgs^{\cO,s}(C,D)$. To show that $\Higgs^{\cO,s}(C,D)=L$, it is enough to prove that $\Higgs^{\cO,s}(C,D)$ is closed in $L$ since $L$ is connected by definition.

By Theorem~\ref{thm:connected HiggsO}, the closure of $\Higgs^{\cO,s}(C,D)$ is $\Higgs^{\ov\cO,s}(C,D)$. If
\[
L\cap\left(\Higgs^{\ov\cO,s}(C,D)\setminus \Higgs^{\cO,s}(C,D)\right)=\emptyset,
\]
then $\Higgs^{\cO,s}(C,D)$ is closed in $L$. Suppose otherwise, and let $\mathcal E$ be an element of this intersection. Then there is a stratum $\cO'\subset \ov\cO\setminus\cO$ such that $\Higgs^{\cO',s}(C,D)$ contains $\mathcal E$. Since $\cO'$ is a boundary stratum, we have $\dim\cO'<\dim\cO$. Therefore
\[
\dim \Higgs^{\cO',s}(C,D)<\dim \Higgs^{\cO,s}(C,D)=\dim L.
\]
By Corollary~\ref{cor:characteristic-distribution}, the rank of the Poisson map at $\mathcal E$ is $\dim \Higgs^{\cO',s}(C,D)$, which is strictly less than $\dim L$. This contradicts the fact that the rank of the Poisson map is constant along the symplectic leaf $L$ and equals $\dim L$. Hence $L$ does not meet the boundary, so $\Higgs^{\cO,s}(C,D)$ is closed in $L$.
\end{proof}

\section{$\vec\xi$-parabolic Higgs bundles}\label{sec:xi-parabolic}
In this section, we discuss the geometry of the moduli space of $\vxi$-parabolic Higgs bundles.
\subsection{$\vec\xi$-parabolic Higgs bundles and their moduli}\label{sec:xi-par moduli}

\begin{definition}\label{def:xi-par}
Fix local flag types $\um_i=(m_{i,1},\dots,m_{i,\ell_i})$ and local eigenvalue data $\uxi_i=(\xi_{i,1},\dots,\xi_{i,\ell_i})$ at each marked point $p_i\in D$ satisfying the residue condition ~\eqref{eq:trace-compatibility}. A parabolic Higgs bundle on $(C,D)$ consists of the following data:
\begin{itemize}
	\item a vector bundle $E$ on $C$ of rank $r$ and degree $d$;
	\item for each $i$, a flag
	\[
	0=E_{p_i}^{\ell_i}\subset E_{p_i}^{\ell_i-1}\subset \cdots \subset E_{p_i}^{0}=E_{p_i},
	\qquad \dim(E_{p_i}^{a-1}/E_{p_i}^{a})=m_{i,a};
	\]
	\item a meromorphic Higgs field $\Phi\in H^0(C,\End(E)\otimes K_C(D))$ such that each residue preserves the corresponding flag.
\end{itemize}
In particular, we say that a parabolic Higgs bundle is a $\vxi$-parabolic Higgs
bundle if its residues satisfy the following condition
	\begin{equation*}\label{eq:xipar}
	\gr_a(\res_{p_i}\Phi)=\xi_{i,a}\id
	\qquad \text{for all }i,a.
	\end{equation*}
\end{definition}
\begin{rem}
The notion of $\vxi$-parabolic Higgs bundles is essentially a special case of parabolic Higgs bundles with fixed residue diagrams in the sense of Simpson \cite{simpson-noncompact} (see Example~\ref{exmp:xipar}). It has also appeared in several closely related forms in the literature; see \cite{BGPM2020,Shu2025DSP,Yae2026StarShaped}.

\end{rem}
\begin{example}
	When $\vec\xi=0$, the residue acts trivially on each graded piece, so one recovers the usual notion of a strongly parabolic Higgs bundle. 
\end{example}

Choose parabolic weights $
0\le \alpha_{i,1}<\alpha_{i,2}<\cdots<\alpha_{i,\ell_i}<1$ for $1 \leq i \leq n$. If $F\subset E$ is a saturated $\Phi$-invariant proper subbundle, the induced flag on $F_{p_i}$ is defined by $F_{p_i}^a:=F_{p_i}\cap E_{p_i}^a.$ Its parabolic degree is
\[
\pardeg_{\vec\alpha}(F):=\deg(F)+\sum_{i=1}^n\sum_{a=1}^{\ell_i}\alpha_{i,a}\dim(F_{p_i}^{a-1}/F_{p_i}^{a}),
\]
and the parabolic slope is
\[
\parmu_{\vec\alpha}(F):=\frac{\pardeg_{\vec\alpha}(F)}{\rk(F)}.
\]
A parabolic Higgs bundle is called $\vec\alpha$-semistable (resp. $\vec\alpha$-stable) if for every proper nonzero $\Phi$-invariant saturated subbundle $F\subset E$ one has
\[
\parmu_{\vec\alpha}(F) \leq \,\parmu_{\vec\alpha}(E) \qquad (\text{resp. }\parmu_{\vec\alpha}(F) < \,\parmu_{\vec\alpha}(E)).
\]
Throughout, we choose $\vec\alpha$ to be generic, so that semistability is
equivalent to stability. We denote by
$\Higgs^{\mathrm{par}}(C,D,\vec m,\vec\alpha)$
(resp. $\Higgs^{\vec\xi\text{-par}}(C,D,\vec m,\vec\alpha)$) the moduli space
of $\vec\alpha$-semistable parabolic Higgs bundles
(resp. $\vec\xi$-parabolic Higgs bundles). Note that $\Higgs^{\vxi\text{-par}}(C,D,\vm,\va)$ is closed in $\Higgs^{\mathrm{par}}(C,D,\vec m,\vec\alpha)$. 

There is a well-defined Hitchin map
\[
h:\Higgs^{\text{par}}(C,D,\vec m,\vec\alpha)
\longrightarrow
A:=\bigoplus_{j=1}^r H^0(C,K_C(D)^{\otimes j})
\]
defined by the coefficients of the characteristic polynomial of the Higgs field $\Phi$. When we restrict $h$ to
$\Higgs^{\vec\xi\text{-par}}(C,D,\vec m,\vec\alpha)$, its image can be described
more precisely using the spectral correspondence of \cite[Section~3]{LeeLee2025}. 

Recall that in \cite[Section 3,4]{LeeLee2025}, we constructed a proper surface $Z_{\vec\xi}$ as an iterative sequence of blow-ups $f:Z_{\vec\xi}\to \mathbb P(K_C(D)\oplus\mathcal O_C)$ over the ruled surface $\mathbb P(K_C(D)\oplus\mathcal O_C)$. To a $\vec\xi$-parabolic Higgs bundle $(E,E_D^\bullet,\Phi)$, we associate a pure one-dimensional sheaf $\mathcal F$ on $Z_{\vec\xi}$. Its support belongs to the linear system $|\Sigma(\vec m)_{\vec\xi}|$, where $\Sigma(\vec m)_{\vec\xi}$ is the divisor class on $Z_{\vec\xi}$ determined by the combinatorial data $(\vec m,\vec\xi)$. Moreover, the support is disjoint from the strict transform of $f^{-1}(D)$ and from the horizontal divisor at infinity. We denote this open surface by $S_{\vec\xi}\subset Z_{\vec\xi}$.

The curves in $|\Sigma(\vec m)_{\vec\xi}|$ satisfying this disjointness condition form an open affine subspace, denoted by $B(\vec m)_{\vec\xi}$, of dimension 
\begin{equation}\label{eq:dim bmxi}
    \dim \bmxi
	=
	r^2(g-1)+1+\frac{1}{2}\sum_{i=1}^n\left(r^2-\sum_{a=1}^{\ell_i}m_{i,a}^2\right).
\end{equation}
There is a natural embedding $B(\vec m)_{\vec\xi}\hookrightarrow A$, obtained by sending a spectral curve to the coefficients of its defining characteristic equation. Thus the Hitchin map factors through $B(\vec m)_{\vec\xi}$, and we write
\begin{equation}\label{eq:Hit xi-par}
h(\vec m)_{\vxi}:\Higgs^{\vec\xi\text{-par}}(C,D,\vec m,\vec\alpha)\to B(\vec m)_{\vec\xi}.
\end{equation}

Furthermore, after choosing an appropriate stability condition on compactly supported pure one-dimensional sheaves on $S_{\vec\xi}$, the spectral correspondence of \cite[Section 4]{LeeLee2025} gives a closed embedding
\[
Q_{\vec\xi}:\Higgs^{\vec\xi\text{-par}}(C,D,\vec m,\vec\alpha)\hookrightarrow \mathcal M_{S_{\vec\xi}},
\]
where $\mathcal M_{S_{\vec\xi}}$ denotes the moduli space of stable compactly supported pure one-dimensional sheaves on $S_{\vec\xi}$ with fixed numerical class $\Sigma(\vec m)_{\vec\xi}$. Under this embedding, the Hitchin map agrees with the Fitting-support map $\operatorname{Supp}:\mathcal M_{S_{\vec\xi}}\to B(\vec m)_{\vec\xi}$. Equivalently, we have the following commutative diagram:
\begin{equation}\label{eq:diag sp}
	    \begin{tikzcd}
				\Higgs^{\xi\text{-par}}(C,D,\vec{m}, \vec{\alpha}) \arrow[r, hook, "Q_{\vec{\xi}}"] \arrow[d, "h"] & \cM_{S_{\vec{\xi}}}  \arrow[d, "\mathrm{Supp}"]\\
				A & B(\vec{m})_{\vec{\xi}} \arrow[l, hook']
			\end{tikzcd}
	\end{equation}
			In particular, over the integral locus $B^{\mathrm{int}}(\vec m)_{\vec\xi} \subset B(\vec m)_{\vec\xi}$ where the corresponding support curve in $S_{\vec{\xi}}$ is integral, the map $Q_{\vec\xi}$ becomes isomorphism by identifying a Hitchin fiber with a compactified Jacobian of the corresponding integral spectral curve in $S_{\vec\xi}$. This leads to the following proposition.

\begin{proposition}\label{prop:xi-par sm.irr}
	Suppose that $B^{\mathrm{int}}(\vec m)_{\vec\xi}$ is nonempty. Then the Hitchin map 
	\[
	h(\vec{m})_{\vxi}:\Higgs^{\vec\xi\text{-par}}(C,D,\vec m,\vec\alpha)\to B(\vec m)_{\vec\xi}
	\]
	is proper and surjective.
\end{proposition}

\begin{proof}
	 Properness of $h(\vec m)$ follows from the closed embedding into the proper support morphism on $\cM_{S_{\vec\xi}}$ (see \cite[Proposition 4.19]{LeeLee2025}).

	It remains to prove surjectivity. Since $B(\vec m)_{\vec\xi}$ is an affine space, it is irreducible; hence the nonempty open subset $B^{\mathrm{int}}(\vec m)_{\vec\xi}$ is dense. Let $b\in B^{\mathrm{int}}(\vec m)_{\vec\xi}$ and let $\Sigma_b\subset S_{\vec\xi}$ be the corresponding integral spectral curve. Choose a rank-one torsion-free sheaf on $\Sigma_b$ of the numerical class prescribed by the spectral correspondence. Via the isomorphism $Q_{\vec{\xi}}$ over the integral locus, this sheaf determines a point of $\Higgs^{\vec\xi\text{-par}}(C,D,\vec m,\vec\alpha)$ mapping to $b$. Thus the image of $h$ contains $B^{\mathrm{int}}(\vec m)_{\vec\xi}$. Since $h$ is proper, its image is closed; therefore it equals all of $B(\vec m)_{\vec\xi}$.

\end{proof}

In fact, the diagram \eqref{eq:diag sp} and the Hitchin map \eqref{eq:Hit xi-par} naturally sit in a smooth family over the parameter space of $\vec{\xi}$, defined by
\[
\nmd:=\left\{\vec{\xi}=(\uxi_1, \cdots, \uxi_n) \mid \sum_{a,i}m_{i,a}\xi_{i,a}=0 \right\}.\]  In other words, there exists a smooth family of moduli spaces of $\vec{\alpha}$-semistable $\vec{\xi}$-parabolic Higgs bundles
\[
P:\bm{{\m{H}}}(\vec{m}) \to \nmd
\] 
whose fiber at $\vec{\xi}$ is $\Higgs^{\xi\text{-par}}(C,D,\vec{m}, \vec{\alpha})$. For simplicity, we write it as $\mathcal{H}(\vec m)_{\vec \xi}$. Also, we have a family version of diagram \eqref{eq:diag sp},
\[\begin{tikzcd}
         \bm{\m{H}}(\vm) \arrow[r, hook, "\bm{Q}"]\arrow[d] & \bm{\m{M}}(\vm) \arrow[d]\\
        \bm{A}(\vm) & \bm{B}(\vm) \arrow[l, hook', "\iota"]
\end{tikzcd}\]
and a family of Hitchin map ${\bm h}(\vec{m}):\bm{\m{H}}(\vm)\to \bm{B}(\vm)$ over $\nmd$.

Recall that a smooth quasi-projective variety $X$ is called semi-projective if there exists a $\C^*$-action on $X$ such that the fixed point locus $X^{\C^*}$ is proper, and for every $x \in X$, $\lim_{\lambda \to 0}\lambda \cdot x$ exists (see \cite[Definition 1.1]{HauselRodriguez2013})

\begin{lem}\label{lem:semiprojective}
Suppose that $B^{\mathrm{int}}(\vec m)_{\vec\xi}$ is non-empty for every $\vec\xi \in \nmd$. Then $\bm{\m{H}}(\vm)$ is semi-projective.
\end{lem}

\begin{proof}
Since the morphism $P:\bm{\m{H}}(\vm) \to \nmd$ is smooth and $\nmd$ is smooth over $\C$, the source $\bm{\m{H}}(\vm)$ is smooth. Consider a $\C^*$-action on $\bm{{\m{H}}}(\vec{m})$ given by $\lambda \cdot (E,E_D^\bullet, \Phi) \mapsto (E,E_D^\bullet, \lambda\Phi)$. We can also define a $\C^*$-action on $\nmd$ by scaling. Then the morphism $P$ is $\C^*$-equivariant. To show that $\bm{\m{H}}(\vm)$ is semi-projective, it is enough to show that ${\bm{\m{H}}(\vm)}^{\C^*}$ is proper and the limit $\lim_{\lambda \to 0} \lambda \cdot (E,\Phi,E^\bullet_D)$ exists for any $(E,E_D^\bullet, \Phi)$. 

First, for properness of ${\bm{\m{H}}(\vm)}^{\C^*}$, note that 
    \[
    \bm{\m{H}}(\vm)^{\C^*}=  \m{H}(\vm)_{\vec{0}}^{\C^*}\subset h(\vm)^{-1}_{\vec{0}}(0)\subset \m{H}(\vm)_{\vec{0}} 
    \]
where $h(\vm)_{\vec{0}}: \m{H}(\vm)_{\vec{0}}\to B(\vm)_{\vec{0}}$ is the restriction map over $\vec{0} \in \nm$ and $0\in B(\vm)_{\vec{0}}$ is the origin of the vector space $B(\vm)_{\vec{0}}$. By the properness of $h(\vm)_{\vec{0}}$, the closed subset $\m{H}(\vm)_{\vec{0}}^{\C^*}$ is proper and so is $\bm{\m{H}}(\vm)^{\C^*}$.

For the existence of limit, let $(E,\Phi,E^\bullet_D)\in \mathcal{H}(\vm)_{\vxi}$. Consider the image $L^\circ$ of 
        \[ \C^*\times \{(E,E^\bullet_D, \Phi)\} \hookrightarrow \C^*\times \bm{\m{H}}(\vm)\to \bm{\m {H}}(\vm)  \]
        Clearly, $h(\vm)(L^\circ)\subset \bm{B}(\vm)$ has a limit at $0\in B(\vm)_{\vec{0}}\subset \bm{B}(\vm).$ Since $h(\vm)$ is proper, it follows that $L^\circ $ also has a limit i.e. $\lim_{\lambda \to 0} \lambda \cdot (E,\Phi,E^\bullet_D)\in \m{H}(\vm)_0$ exists (cf. \cite[Theorem 5.2]{Yokogawa-infinitesimal})
    \end{proof}

Using Lemma \ref{lem:semiprojective}, we prove the following connectedness result. 

\begin{proposition}\label{prop:connected}
Under the same assumption, the moduli space $
\Higgs^{\vec\xi\text{-par}}(C,D,\vec m,\vec\alpha)$ is connected. 
\end{proposition}
\begin{proof}
    Since the morphism $P:\bm{\m{H}}(\vm)\to \nmd$ is $\C^*$-equivariant with positive weights on the base, we can apply the results in \cite[Theorem 7.2.1]{HLRV2011Ari} for any complex line in $\nmd$ passing through the origin. This implies that every fiber has the same cohomology. In particular $H^0(\m{H}(\vec{m})_0) \cong H^0(\m{H}(\vec{m})_{\vec{\xi}})$. By the same argument used in \cite[Theorem 5.2]{Yokogawa-infinitesimal}, we get $\m{H}(\vec{m})_0$ is connected. Therefore, the conclusion follows. 
\end{proof}
For later use, we record the flatness of the Hitchin map.

\begin{lem}\label{l:flatness xi-par Hitchin}
	With the notation as above, the morphism
	\[
	{\bm h}(\vec m):\bm{\m H}(\vm)\to \bm B(\vm)
	\]
	is flat. In particular, for every $\vxi\in \nm$, the induced Hitchin map
	\[
	h(\vec m)_{\vxi}:
	\Higgs^{\vec\xi\text{-par}}(C,D,\vec m,\vec\alpha)
	\to
	B(\vec m)_{\vec\xi}
	\]
	is flat.
\end{lem}

\begin{proof}
	The second statement follows from the first by base change along the inclusion
	$B(\vec m)_{\vec\xi}\subset \bm B(\vm)$. Thus we prove the flatness of
	$\bm h(\vm)$.

	We use miracle flatness. Since $\bm{\m H}(\vm)$ is smooth and $\bm B(\vm)$ is
	an affine space, it is enough to show that every fiber of $\bm h(\vm)$ has the
	expected dimension. By the $\C^*$-equivariance of $\bm h(\vm)$, the fibers over
	points in the same $\C^*$-orbit are isomorphic. Moreover, the $\C^*$-orbit closure of
	any point $(b,\vxi)\in \bm B(\vm)=\bmxi\times \nm$ contains $(0,\vec 0)$.
	Therefore, by upper semicontinuity of fiber dimension,
	\[
	\dim \bm h(\vm)^{-1}(b,\vxi)
	\leq
	\dim \bm h(\vm)^{-1}(0,\vec 0).
	\]
	By \cite[Theorem~6.8]{Suwangwen2022}, the central fiber has the expected
	dimension:
	\[
	\dim \bm h(\vm)^{-1}(0,\vec 0)
	=
	\dim \m H(\vm)_0-\dim B(\vm)_0.
	\]
	On the other hand, by the general lower bound for fiber dimensions,
	\[
	\dim \bm h(\vm)^{-1}(b,\vxi)=\dim h(\vm)_{\vxi}^{-1}(b)
	\geq
	\dim \m H(\vm)_{\vxi}-\dim B(\vm)_{\vxi}.
	\]
	By Lemma~\ref{lem:dim} and \eqref{eq:dim bmxi}, the dimension formula of $\bmxi$, the right-hand side is
	independent of $\vxi$ and equals $
	\dim \m H(\vm)_0-\dim B(\vm)_0$. Hence every fiber of $\bm h(\vm)$ has the expected dimension, and the result follows from the miracle flatness. 
\end{proof}

\subsection{The deformation theory}\label{sec:xi-par deformation}

We study the deformation-theoretic construction of the holomorphic symplectic form on
$\Higgs^{\vec\xi\text{-par}}(C,D,\vec m,\vec\alpha)$; see \cite{biswasramanan,Bottacin1995symplectic,Yokogawa-infinitesimal, logaresMartens}.

Let $\mathcal E=(E,E_D^\bullet,\Phi)$ be a stable $\vec\xi$-parabolic Higgs bundle. Define
\[
\begin{aligned}
    \PEnd(\mathcal E)&:=\{f\in \End(E)\mid f_{p_i}(E_{p_i}^a)\subset E_{p_i}^a\ \text{for all }a,i\}, \\
    \SPEnd(\mathcal E)&:=\{f\in \End(E)\mid f_{p_i}(E_{p_i}^{a-1})\subset E_{p_i}^{a}\ \text{for all }a,i\}.
\end{aligned}
\]
Consider the two-term complex
\[
C_{\vxi\text{-par},\mathcal E}^\bullet:
\qquad
\PEnd(\mathcal E)
\xrightarrow{\ [\Phi,\cdot]\ }
\SPEnd(\mathcal E)\otimes K_C(D),
\]
placed in degrees $0,1$.

\begin{lemma}\label{lem:tangent-hypercoh}
	There is a canonical identification
	$T_{\mathcal E}\Higgs^{\vec\xi\text{-par}}(C,D,\vec m,\vec\alpha)\cong \HH^1(C_{\vxi\text{-par},\mathcal E}^\bullet)$.
\end{lemma}

\begin{proof}
	An infinitesimal deformation consists of a deformation of the bundle and the quasi-parabolic flags, together with a deformation $\Phi+\varepsilon\dot\Phi$ of the Higgs field preserving the deformed flags and the fixed graded residue eigenvalues. Infinitesimal automorphisms preserving the flags are sections of $\PEnd(\mathcal E)$, and their infinitesimal action on Higgs fields is $f\mapsto[\Phi,f]$. Since the graded residue eigenvalues are fixed, the induced infinitesimal deformation on each graded piece vanishes; equivalently, $\dot\Phi$ lies in $\SPEnd(\mathcal E)\otimes K_C(D)$. The standard \v Cech deformation calculation (see e.g. \cite[Section 5]{Bottacin1995symplectic}) therefore gives the stated identification.
\end{proof}

The trace pairing induces a natural sheaf morphism
\[
\PEnd(\mathcal E)\otimes \SPEnd(\mathcal E)\to \mathcal O_C(-D),
\qquad
(f,g)\longmapsto \mathrm{tr}(fg).
\]
which is nondegenerate. Thus we have canonical identifications
$\PEnd(\mathcal E)^\vee \cong \SPEnd(\mathcal E)\otimes \mathcal O_C(D)$ and
$\SPEnd(\mathcal E)^\vee \cong \PEnd(\mathcal E)\otimes \mathcal O_C(D)$, where $(-)^\vee$ denotes the $\mathcal O_C$-linear dual. Then the Serre-dual complex of $C^\bullet_{\vxi\text{-par}}$, $D(C^\bullet_{\vxi\text{-par},\mathcal E})=(C^\bullet_{\vxi\text{-par},\mathcal E})^\vee \otimes K_C[-1]$, is given by 
\[
D(C^\bullet_{\vxi\text{-par},\mathcal E}):
\qquad
\PEnd(\mathcal E)
\xrightarrow{\ -[\Phi,\cdot]\ }
\SPEnd(\mathcal E)\otimes K_C(D),
\]
placed in degrees $0,1$. 
\begin{lemma}\label{lem:dim}
	The dimension of $\Higgs^{\vec\xi\text{-par}}(C,D,\vec m,\vec\alpha)$ is
	\[
	\dim \Higgs^{\vec\xi\text{-par}}(C,D,\vec m,\vec\alpha)
	=
	2r^2(g-1)+2+\sum_{i=1}^n\left(r^2-\sum_{a=1}^{\ell_i}m_{i,a}^2\right).
	\]
	Equivalently,
	\[
	\dim \Higgs^{\vec\xi\text{-par}}(C,D,\vec m,\vec\alpha)
	=
	2r^2(g-1)+2+\sum_{i=1}^n\dim\cO_i,
	\]
	where $\cO_i=\cO(\um_i,\uxi_i)$.
\end{lemma}

\begin{proof}
	The argument goes parallel to the proof of Proposition \ref{prop:smooth-HiggsO}. By Lemma \ref{lem:tangent-hypercoh}, the tangent space is $\HH^1(C_{\text{par},\mathcal E}^\bullet)$. Since $\mathcal{E}$ is $\va$-stable, $\HH^0(C_{\text{par},\mathcal E}^\bullet)\cong \C$, and Serre duality gives $\HH^2(C_{\text{par},\mathcal E}^\bullet)\cong \C$. Hence
	$\dim \HH^1(C_{\text{par},\mathcal E}^\bullet)=2-\chi(C_{\text{par},\mathcal E}^\bullet)$. At each marked point $p_i$,
	\[
	\dim\PEnd(\mathcal E)_{p_i}=\sum_{a\leq b}m_{i,a}m_{i,b},
	\qquad
	\dim\SPEnd(\mathcal E)_{p_i}=\sum_{a<b}m_{i,a}m_{i,b}.
	\]
	Using the exact sequences defining $\PEnd(\mathcal E)$ and $\SPEnd(\mathcal E)$ as subsheaves of $\End(E)$, together with Riemann--Roch for $\End(E)$ and $\End(E)\otimes K_C(D)$, one obtains
	\[
	\dim \HH^1(C_{\text{par},\mathcal E}^\bullet)
	=
	2r^2(g-1)+2+\sum_{i=1}^n\left(r^2-\sum_{a=1}^{\ell_i}m_{i,a}^2\right).
	\]
	The final expression follows from the standard formula
	$\dim \cO_i=r^2-\sum_a m_{i,a}^2$ for the conjugacy class associated to $(\um_i,\uxi_i)$.
\end{proof}

Next, we describe a symplectic form on $\Higgs^{\vxi\text{-par}}(C,D,\vm,\va)$. Serre duality and Lemma \ref{lem:tangent-hypercoh} imply that $T^*_\mathcal E\Higgs^{\vxi\text{-par}}(C,D,\vm,\va) \cong \mathbb{H}^1(D(C^\bullet_{\vxi\text{-par},\mathcal E}))$. Moreover, there is a canonical morphism of complexes $D(C^\bullet_{\vxi\text{-par},\mathcal E}) \to C^\bullet_{\vxi\text{-par},\mathcal E}$ given by the identity up to a sign:
\begin{equation}\label{e:xipar serredual map}
\begin{tikzcd}
    \PEnd(\mathcal E)
    \arrow[r, "{-[\Phi,\cdot]}"]
    \arrow[d, hook, "id"]
    &
    \SPEnd(\mathcal E)\otimes K_C(D)
    \arrow[d, hook, "-id \otimes id_{K_C(D)}"]
    \\
    \PEnd(\mathcal E)
    \arrow[r, "{[\Phi,\cdot]}"]
    &
    \SPEnd(\mathcal E)\otimes K_C(D).
\end{tikzcd}
\end{equation}
The induced map on the hypercohomology $\mathbb{H}^1(-)$ gives an isomorphism 
\begin{equation}\label{e:xipar serre dual}
    \pi^\sharp_{\vxi\text{-par},\mathcal{E}}:T^*_\mathcal E\Higgs^{\vxi\text{-par}}(C,D,\vm,\va) \xrightarrow{\cong} T_{\mathcal E}\Higgs^{\text{par}}(C,D,\vm,\va).
\end{equation}
Equivalently, we have a nondegenerate  bilinear form at each tangent space. 
\[
\omega^{\vxi\text{-par}}_{\mathcal E}:
\HH^1(C_{\vxi\text{-par},\mathcal E}^\bullet)\times
\HH^1(C_{\vxi\text{-par},\mathcal E}^\bullet)
\longrightarrow
H^1(C,K_C)\cong \C.
\]

To show that the induced $\omega^{\vxi\text{-par}}$ is a symplectic form, we compare it with the Poisson structure on the moduli space of $\va$-semistable parabolic Higgs bundles $\Higgs^{\text{par}}(C,D,\vm,\va)$ constructed in \cite[Theorem 3.1]{logaresMartens}. At a point $\mathcal E$, the deformation complex of the ambient parabolic Higgs moduli space is
\[
C_{\text{par},\mathcal E}^\bullet:
\qquad
\PEnd(\mathcal E)
\xrightarrow{\ [\Phi,\cdot]\ }
\PEnd(\mathcal E)\otimes K_C(D),
\]
placed in degrees $0,1$. The Serre-dual complex $D(C_{\text{par},\mathcal E}^\bullet)$ is given by
\[
D(C_{\text{par},\mathcal E}^\bullet):
\qquad
\SPEnd(\mathcal E)
\xrightarrow{\ -[\Phi,\cdot]\ }
\SPEnd(\mathcal E)\otimes K_C(D),
\]
again placed in degrees $0,1$. With this definition, it is known that 
\[
T_\mathcal E\Higgs^{\text{par}}(C,D,\vm,\va)
\cong
\HH^1(C_{\text{par},\mathcal E}^\bullet),
\qquad
T^*_\mathcal E\Higgs^{\text{par}}(C,D,\vm,\va)
\cong
\HH^1(D(C_{\text{par},\mathcal E}^\bullet)).
\]
\noindent The Poisson map
\[
\pi^\sharp_{\text{par},\mathcal E}:
T^*_\mathcal E\Higgs^{\text{par}}(C,D,\vm,\va)
\longrightarrow
T_\mathcal E\Higgs^{\text{par}}(C,D,\vm,\va)
\]
is induced by $\HH^1(-)$ applied to the natural inclusion of complexes $
D(C_{\text{par},\mathcal E}^\bullet)
\longrightarrow
C_{\text{par},\mathcal E}^\bullet$, given by the commutative diagram
\begin{equation}\label{e:par serrredual map}
    \begin{tikzcd}
    \SPEnd(\mathcal E)
    \arrow[r, "{-[\Phi,\cdot]}"]
    \arrow[d, hook, "id"]
    &
    \SPEnd(\mathcal E)\otimes K_C(D)
    \arrow[d, hook, "-id \otimes id_{K_C(D)}"]
    \\
    \PEnd(\mathcal E)
    \arrow[r, "{[\Phi,\cdot]}"]
    &
    \PEnd(\mathcal E)\otimes K_C(D).
\end{tikzcd}
\end{equation}

Now observe that this inclusion factors through the $\vxi$-parabolic deformation complex, and the following diagram commutes
\[
\begin{tikzcd}
    D(C_{\text{par},\mathcal E}^\bullet) \arrow[r,hook] \arrow[d] & D(C_{\vxi\text{-par},\mathcal E}^\bullet) \arrow[d] \\
    C_{\text{par},\mathcal E}^\bullet& C_{\vxi\text{-par},\mathcal E}^\bullet \arrow[l, hook']
\end{tikzcd}
\]
where the horizontal maps are the canonical inclusions and vertical maps are \eqref{e:par serrredual map}(left) and \eqref{e:xipar serredual map}(right). By taking the hypercohomology $\mathbb{H}^1(-)$, we obtain the following commutative diagram 
\[
\begin{tikzcd}
T^*_\mathcal{E}\Higgs^{\text{par}}(C,D,\vm,\va) \arrow[r,two heads] \arrow[d, "\pi^\sharp_{\text{par},\mathcal E}"] & T^*_\mathcal{E}\Higgs^{\vxi\text{-par}}(C,D,\vm,\va) \arrow[d, "\pi^\sharp_{\vxi\text{-par},\mathcal{E}}"] \\
T_\mathcal{E}\Higgs^{\text{par}}(C,D,\vm,\va) & T_\mathcal{E}\Higgs^{\vxi\text{-par}}(C,D,\vm,\va) \arrow[l, hook']
\end{tikzcd}
\]
where the right-hand side vertical map is an isomorphism \eqref{e:xipar serre dual}. Here, it follows from a similar argument as in Lemma \ref{lem:inj,surj,co} that the upper (resp. lower) horizontal map is surjective (resp. injective).

By Proposition \ref{prop:symp leaf closed}, we have the following corollary.

\begin{corollary}
    For $\mathcal{E}\in \Higgs^{\vxi\text{-par}}(C,D,\vm,\va),$ the image of the Poisson map $\pi^{\sharp}_{\text{par}, \mathcal{E}}$ is precisely  $T_{\mathcal{E}}\Higgs^{\vxi\text{-par}}(C,D,\vm,\va)$. Moreover, the induced nondegenerate bivector on $\Higgs^{\vxi\text{-par}}(C,D,\vm,\va)$ defines a holomorphic symplectic form $\omega^{\vxi\text{-par}}$. 
\end{corollary}

Note that $\Higgs^{\vxi\text{-par}}(C,D,\vm,\va)$ is closed in $\Higgs^{\text{par}}(C,D,\vm,\va)$ since $\vxi$-parabolic condition is a closed condition. Combining this fact with the connectedness (Proposition \ref{prop:connected}) we conclude that $\Higgs^{\vxi\text{-par}}(C,D,\vm,\va)$ is a symplectic leaf of the Poisson variety $\Higgs^{\text{par}}(C,D,\vm,\va)$. 

\begin{theorem}\label{thm:symplectic leaf xipar}
    The moduli space $\Higgs^{\vxi\text{-par}}(C,D,\vm,\va)$ is a closed symplectic leaf of the Poisson variety $\Higgs^{\text{par}}(C,D,\vm,\va)$. 
\end{theorem}
\begin{proof}
    The proof is similar to Theorem \ref{thm:orbit-stratum-is-leaf}.
    Let $L$ be a symplectic leaf through a point of $\Higgs^{\vxi\text{-par}}(C,D,\vm,\va)$. By definition, we have an open inclusion $\Higgs^{\vxi\text{-par}}(C,D,\vm,\va) \subset L$. Since $L$ is smooth and connected, it is enough to show that the closure of $\Higgs^{\vxi\text{-par}}(C,D,\vm,\va)$ in $L$ is $\Higgs^{\vxi\text{-par}}(C,D,\vm,\va)$. Since its closure in the ambient space $\Higgs^{\text{par}}(C,D,\vm,\va)$ is itself, the claim follows. 
\end{proof}

\section{Main results}
We first recall the notion of an integrable system and prove
Theorem~\ref{thm:xipar} (= Theorem~\ref{thm:holosym and integrable}) in
Section~\ref{sec:integrable systems}. We then prove the remaining main results
in Section~\ref{sec:parlift}.
\subsection{Integrable systems}\label{sec:integrable systems}
In this section, we study integrable system structures for both $\Higgs^{\cO,s}(C,D)$ and $\Higgs^{\text{par}}(C,D,\vm,\va)$. There are many different conventions for the definition of an integrable system. We first set this clearly.

\begin{definition}\label{def:integrable system}
    Let $X$ be a symplectic variety and $B$ be a smooth variety such that $\dim X=2\dim B$.  A morphism $h:X \to B$ is called an algebraically completely integrable system (simply called integrable system) if $h$ is a proper flat morphism whose generic smooth fibers are Lagrangian abelian varieties (see \cite[Definition 2.9]{DonagiMarkman1996}). 
\end{definition}

\begin{remark}
   In general, one may not need to require the properness of $h$. On the other hand, for the base $B$, it is often required to be affine (See \cite[Definition 1]{boalch2018wild}).
\end{remark}

More generally, one can consider the case where $X$ is Poisson, not necessarily symplectic, and $\dim X < 2\dim B$. In this case, a generic fiber is required to be coisotropic, instead of Lagrangian. We call such $h:X \to B$ a Poisson integrable system. 

Since $\Higgs^{\vxi\text{-par}}(C,D,\vm,\va)$ is a symplectic leaf of $\Higgs^{\text{par}}(C,D,\vm,\va)$ (Theorem \ref{thm:symplectic leaf xipar}) and $\Higgs^{\cO,s}(C,D)$ is a symplectic leaf of $\Higgs^{s}(C,D)$ (Theorem \ref{thm:orbit-stratum-is-leaf}), it is natural to ask whether the restriction of the Hitchin map $h$ gives an integrable system. We first consider the $\vxi$-parabolic case. 
\begin{example}\label{ex:meromorphic integrable}
    Both the parabolic Hitchin system $h:\Higgs^{\text{par}}(C,D,\vm,\va) \to A$ and  the meromorphic Hitchin system $h:\Higgs^{s}(C,D) \to A$ are known to be Poisson integrable systems (see \cite{logaresMartens} for parabolic one and \cite{markmanspectral} for meromorphic one). However, it is worth noting that these results only imply that the generic symplectic leaves (corresponding to maximal orbits $\cO$) are algebraically integrable integrable system.  
    
\end{example}

\begin{theorem}\label{thm:holosym and integrable}
	 The Hitchin map
	$h(\vm)_{\vxi}:\Higgs^{\vec\xi\text{-par}}(C,D,\vec m,\vec\alpha)\to B(\vec m)_{\vec\xi}$
	is an algebraically completely integrable system.
\end{theorem}

\begin{proof}
	Recall that we have a symplectic form $\omega^{\vxi\text{-par}}$ on the smooth moduli space $\Higgs^{\vec\xi\text{-par}}(C,D,\vec m,\vec\alpha)$.

	It remains to show that a generic fiber of the Hitchin map
	$h(\vm)_{\vxi}:\Higgs^{\vec\xi\text{-par}}(C,D,\vec m,\vec\alpha)\to B(\vec m)_{\vec\xi}$
	is a Lagrangian abelian variety. Let $b\in B^{\mathrm{sm}}(\vec m)_{\vec\xi}$ be a point such that the corresponding spectral curve $\Sigma_b\subset S_{\vec\xi}$ is smooth. By the spectral correspondence recalled above, the map $Q_{\vec\xi}:\Higgs^{\vec\xi\text{-par}}(C,D,\vec m,\vec\alpha)\hookrightarrow M_{S_{\vec\xi}}$ identifies the Hitchin fiber $h(\vm)_{\vxi}^{-1}(b)$ with the Jacobian $\Pic^{d'}(\Sigma_b)$ for the appropriate degree $d'$. Therefore, a generic fiber is an abelian variety. 

    To check Lagrangian, first note that Lemma \ref{lem:dim} and the description of the Hitchin base \eqref{eq:dim bmxi} imply that
	$\dim \Higgs^{\vec\xi\text{-par}}(C,D,\vec m,\vec\alpha)=2\dim B(\vec m)_{\vec\xi}$. Therefore, it suffices to show that a generic fiber $h(\vm)_{\vxi}^{-1}(b)$ is coisotropic. This follows from a minor modification of \cite[Proposition 3.12, 3.13]{logaresMartens}. Indeed, we replace the spectral curve $X_s$ used in \textit{loc.cit.} by $\Sigma_b$. Then the same argument in the proof implies that $h(\vm)_{\vxi}^{-1}(b)$ is coisotropic.

    Finally, since the Hitchin map is proper (Proposition \ref{prop:xi-par sm.irr}) and flat (Lemma \ref{l:flatness xi-par Hitchin}), it is an algebraically completely integrable system.

\end{proof}

Next, we move on to $\Higgs^{\cO,s}(C,D)$. We first restrict the base of the Hitchin map as done in $\vxi$-parabolic Higgs bundles. For each orbit $\mathcal O_i$, let $(\um_i,\uxi_i)$ be the corresponding combinatorial data. By Theorem~\ref{thm:orbit-correspondence}, we have
$X_{\um_i,\uxi_i}(E_{p_i})=\overline{\mathcal O_i}$.
Therefore, if $\res_{p_i}(\Phi)\in\overline{\mathcal O_i}$, then there exists a flag $E_{p_i}^\bullet$ of type $(\um_i,\uxi_i)$ compatible with $\res_{p_i}(\Phi)$. Choosing such a flag for each marked point $p_i$, we obtain a $\vec\xi$-parabolic Higgs bundle $(E,E_D^\bullet,\Phi)$ where $\vec\xi=(\uxi_1, \cdots, \uxi_n)$. We regard as a parabolic lift of the meromorphic Higgs bundle $(E,\Phi)$.

As explained in Section~\ref{sec:xi-par moduli}, the image of the Hitchin map of any such $\vec\xi$-parabolic Higgs bundle lies in the affine subspace $B(\vec m)_{\vec\xi}\subset A$. Since the characteristic polynomial is independent of the choice of compatible flags, it follows that the original meromorphic Higgs bundle $(E,\Phi)$ also satisfies $h(E,\Phi)\in B(\vec m)_{\vec\xi}$. To emphasize that we are considering the meromorphic Hitchin map, we write $B^{\ov\cO}$ for $B(\vm)_{\vxi}$, that is where the locus of meromorphic Higgs bundles with residue condition $\res_D(\Phi)\in\overline{\mathcal O}$ maps to under the Hitchin map. 

\begin{remark}
    A parabolic lift of $(E,\Phi)$ is not unique in general. It is unique only over the open orbit locus, i.e. when $\res_D(\Phi)\in\mathcal O$. If $\res_{p_i}(\Phi)\in\overline{\mathcal O_i}\setminus\mathcal O_i$ for some $i$, then there can be several compatible flags of type $(\um_i,\uxi_i)$ on $E_{p_i}$. This freedom is parametrized by the fiber of
$X_{\um_i,\uxi_i}(E_{p_i})\to\overline{\mathcal O_i}$
over $\res_{p_i}(\Phi)$. Different choices of compatible flags may give different pure one-dimensional sheaves under the spectral construction, but their supports are the same; see \cite[Remark 4.11]{LeeLee2025}. Therefore the choice of parabolic lift does not affect the image of the Hitchin map.

Suppose that $\res_D(\Phi)$ lies in a product of smaller strata $\mathcal O^\gamma\subset\overline{\mathcal O}\setminus\mathcal O$. Then one may also choose the combinatorial data $(\vec m^\gamma,\vec\xi^\gamma)$ corresponding to the stratum $\mathcal O^\gamma$. With this choice, the same argument shows that the Hitchin image lies in $B^{\ov{\cO^\gamma}}:=B(\vec m^\gamma)_{\vec\xi^\gamma}$. Indeed, there is a canonical embedding
$B^{\ov{\cO^\gamma}} \hookrightarrow B^{\ov\cO}$
via taking iterative blow-ups $Z_{\vec\xi} \to Z_{\vec\xi^\gamma}$ (see \cite[Section 3.2]{LeeLee2025}). Hence, after this embedding, we may still regard the Hitchin image as lying in $B^{\ov\cO}$.
\end{remark}
   
\noindent We summarize the relation in the following commutative diagram:
\[
\begin{tikzcd}
    \Higgs^{\cO,ss}(C,D) \arrow[r, hook] \arrow[rd, "h^{\cO}"]&\Higgs^{\ov\cO, ss}(C,D) \arrow[r, hook] \arrow[d, "h^{\ov\cO}"] & \Higgs^{ss}(C,D) \arrow[d, "h"] \\
     &B^{\ov\cO} \arrow[r,hook] & A
\end{tikzcd}
\]
where $h^{\cO}$ and $h^{\ov\cO}$ are the restriction of $h$ to $\Higgs^{\cO,ss}(C,D)$ and $\Higgs^{\ov\cO,ss}(C,D)$, respectively. 

Restricting to the stable locus, one may ask whether $h^{\cO}$ and $h^{\ov\cO}$ are integrable systems in the sense of Definition \ref{def:integrable system}. However, both may not form an integrable system in this sense:
\begin{itemize}
    \item For $h^{\ov\cO}:\Higgs^{\ov\cO,s}(C,D) \to B^{\ov\cO}$, the moduli space $\Higgs^{\ov\cO,s}$ is not smooth in general. 
    \item For $h^{\cO}:\Higgs^{\cO,s}(C,D) \to B^{\ov\cO}$, we have $\dim \Higgs^{\cO,s}(C,D)=2\dim B^{\ov\cO}$,
    but we may lose properness of the morphism $h^{\cO}$. For any element $b$ in $B^{\ov\cO}$ which does not lie in $B^{\ov\cO^\gamma}$ for any boundary stratum $\cO^\gamma$ of $\ov\cO$, the fiber $h^{-1}(b)$ is proper due to Example \ref{ex:meromorphic integrable} and Theorem \ref{thm:connected HiggsO}. However, we may not know whether there exists $(E,\Phi) \in \Higgs^{\cO,s}(C,D)$ whose image lies in some $B^{\ov\cO^{\gamma}}$. 
\end{itemize}

\subsection{Parabolic lifts on the moduli spaces}\label{sec:parlift}

In the previous section, we introduced a parabolic lift of a meromorphic Higgs bundle $(E,\Phi)$ whose residue lies in $\ov\cO$. We promote this on the level of moduli spaces by studying the stability conditions. The key input is that, for sufficiently small parabolic weights $\vec\alpha$, stability of the underlying meromorphic Higgs bundle forces stability of every compatible parabolic lift.

\begin{lemma}\label{lem:small-weights}
There exists $0<\varepsilon_0<1$, depending only on $(r,\vec m)$, such that for any system of parabolic weights $\vec\alpha=(\alpha_{i,a})$ with $0<\alpha_{i,a}<\varepsilon_0$, and for any stable meromorphic Higgs bundle $(E,\Phi)$, every compatible $\vec\xi$-parabolic lift of $(E,\Phi)$ is $\vec\alpha$-stable.
\end{lemma}

\begin{proof}
Let $F\subset E$ be a proper nonzero $\Phi$-invariant subbundle. Since $(E,\Phi)$ is stable,
\[
\delta(F):=\deg(E)\rk(F)-\deg(F)\rk(E)>0.
\]
Since $\delta(F)$ is an integer, $\delta(F)\geq 1$. For a compatible parabolic lift, set
\[
n_{i,a}(F):=\dim(F_{p_i}^{a-1}/F_{p_i}^{a}),
\qquad
0\leq n_{i,a}(F)\leq m_{i,a},
\]
and
\[
w_{\alpha}(F):=\sum_{i,a}\alpha_{i,a}n_{i,a}(F),
\qquad
w_{\alpha}(E):=\sum_{i,a}\alpha_{i,a}m_{i,a}.
\]
The parabolic stability inequality is equivalent to
\[
\delta(F)>\rk(E)w_{\alpha}(F)-\rk(F)w_{\alpha}(E).
\]
Since the integers $n_{i,a}(F)$ are bounded in terms of $\vec m$, there is a constant $M=M(r,\vec m)$ such that
\[
\left|\rk(E)w_{\alpha}(F)-\rk(F)w_{\alpha}(E)\right|
\leq M\cdot \max_{i,a}\alpha_{i,a}
\]
for every proper nonzero $\Phi$-invariant subbundle $F$. Choose $\varepsilon_0<1/M$. Then the right-hand side has absolute value $<1$, while $\delta(F)\geq 1$. Hence the parabolic stability inequality holds for every such $F$.
\end{proof}

Unlike stable objects, it is not clear in general that every semistable meromorphic Higgs bundle admits a compatible semistable or stable $\vec\xi$-parabolic lift. Next, we introduce the map of forgetting parabolic structures. This can be seen more generally as a map between moduli spaces of parabolic Higgs bundles of different flag types (see e.g. \cite[Section 4]{logaresMartens} and  \cite[Section 3.5]{Shu2025DSP}).

\begin{proposition}\label{prop:well-defined-forgetful}
Choose a system of parabolic weights $\vec\alpha$ sufficiently small as in Lemma~\ref{lem:small-weights}. Then, for $\cO=\cO(\vec m,\vec\xi)$, forgetting the flags defines a morphism
\[
F:\Higgs^{\vec\xi\text{-par}}(C,D,\vec m,\vec\alpha)\longrightarrow \Higgs^{\ov\cO,ss}(C,D).
\]
\end{proposition}

\begin{proof}
Let $(E,E_D^\bullet,\Phi)$ be a stable $\vec\xi$-parabolic Higgs bundle. For every proper nonzero $\Phi$-invariant subbundle $F\subset E$, parabolic stability gives
\[
\frac{\deg F+w_{\alpha}(F)}{\rk F}
<
\frac{\deg E+w_{\alpha}(E)}{\rk E}.
\]
Equivalently, $
\deg(E)\rk(F)-\deg(F)\rk(E)>
\rk(E)w_{\alpha}(F)-\rk(F)w_{\alpha}(E)$. 
By the smallness of the weights, the right-hand side has absolute value $<1$. Since the left-hand side is an integer, it is nonnegative. Hence the underlying meromorphic Higgs bundle $(E,\Phi)$ is semistable. The residue condition $\res_{D}(\Phi) \in \ov\cO$ follows from Theorem~\ref{thm:orbit-correspondence}.
\end{proof}

\begin{proposition}\label{prop:F properties}
Under the assumptions of Proposition~\ref{prop:well-defined-forgetful}, the forgetful map $F$ satisfies the following properties.
\begin{enumerate}
\item $\Higgs^{\ov\cO,s}(C,D)\subseteq \operatorname{Im}(F)$.
\item Over $\Higgs^{\ov\cO,s}(C,D)$, the morphism $F$ is projective. Moreover, over $\Higgs^{\cO,s}(C,D)$ it is an isomorphism and is compatible with the symplectic forms $\omega^{\vxi\text{-par}}$ and $\omega^{\cO}$.
\item $F$ is compatible with the Hitchin maps.
\end{enumerate}
\end{proposition}

\begin{proof}
The first statement follows from Theorem~\ref{thm:orbit-correspondence} and Lemma~\ref{lem:small-weights}. For a stable meromorphic Higgs bundle $(E,\Phi) \in \Higgs^{\ov\cO,s}(C,D)$, there exists at least one parabolic lift and Lemma~\ref{lem:small-weights} guarantees that such a lift is $\va$-stable. 

Next we prove projectivity over the stable orbit-closure locus. Fix a stable meromorphic Higgs bundle $(E,\Phi)$ with residues $A_i=\res_{p_i}(\Phi)$ in $\ov\cO_i$ for all $i$. The fiber of $F$ over $(E,\Phi)$ is given by 
\[
\prod_{i=1}^n
\left\{
F_i^\bullet\in \Flag_{\um_i}(E_{p_i})
\ \middle|\
(A_i-\xi_{i,a}\id)(F_i^{a-1})\subset F_i^a
\text{ for all }a
\right\}
\]
because every $\vxi$-parabolic lift is $\va$-stable. This is a closed subvariety of a product of partial flag varieties, hence projective. Globally, these fibers assemble into the relative incidence scheme inside the product of the universal flag bundles over $\Higgs^{\ov\cO,s}(C,D)$. Since the incidence conditions are closed, $F$ is projective over $\Higgs^{\ov\cO,s}(C,D)$. 

In particular, if $(E,\Phi)\in\Higgs^{\cO,s}(C,D)$, there is a unique lift so that $F$ is bijective over $\Higgs^{\cO,s}(C,D)$. Proposition~\ref{prop:dFx-isomorphism} below shows that $dF$ is an isomorphism over $\Higgs^{\cO,s}(C,D)$. Since both sides are smooth, $F$ is \'{e}tale over this locus. Hence, $F$ is an isomorphism over this locus. The compatibility of symplectic structures is also proved in Proposition~\ref{prop:dFx-isomorphism}.
 
Finally, $F$ is compatible with the Hitchin maps because forgetting the parabolic flags does not change the Higgs field, and hence does not change its characteristic polynomial. Therefore the diagram
\[
\begin{tikzcd}
\Higgs^{\vec\xi\text{-par}}(C,D,\vec m,\vec\alpha) \arrow[r, "F"] \arrow[d, "h(\vm)_{\vxi}"'] 
& \Higgs^{\ov\cO,ss}(C,D) \arrow[d, "h^{\ov\cO}"] \\
 B(\vec m)_{\vec\xi} \arrow[r, equal]& B^{\ov\cO}
\end{tikzcd}
\]
commutes.
\end{proof}

To compare deformation complexes $C_{\vxi\text{-par}}^\bullet$ and $C_{\cO}^\bullet$, we need the following linear algebra fact. 
\begin{lemma}\label{lem:local-linear-algebra-injective}
Let $V$ be a finite-dimensional complex vector space, and let
$A\in End(V)$ be an element of the conjugacy class $\cO(\um,\uxi)$. Let $
0=F^\ell\subset F^{\ell-1}\subset\cdots\subset F^1\subset F^0=V$
be the compatible flag of type $(\um,\uxi)$. Set
\[
\begin{aligned}
    P&:=\{\phi\in End(V)\mid \phi(F^a)\subset F^a\text{ for all }b\}, \\
    SP&:=\{\phi\in End(V)\mid \phi(F^{a-1})\subset F^b\text{ for all }a\}.
\end{aligned}
\]
Then $[A,P]=SP$. Moreover, if $\phi\in End(V)$ satisfies $[A,\phi]\in SP$, then $\phi\in P$.
\end{lemma}

\begin{proof}
First, it is clear that $[A,P]\subset SP$. We now compare dimensions. Since the compatible flag of $A$ is canonical, every
endomorphism commuting with $A$ preserves this flag. Thus the centralizer $Z(A)$ of $A$ lies in $P$. Hence
$Z_P(A):=Z(A) \cap P = Z(A)$, and therefore
\[
\dim [A,P]=\dim P-\dim Z_P(A)=\dim P-\dim Z(A).
\]
For the conjugacy class $\cO(\um,\uxi)$, the centralizer dimension is $
\dim Z(A)=\sum_{a=1}^{\ell}m_a^2$. 
On the other hand, $P/SP$ is naturally identified with
$\bigoplus_{a=1}^{\ell} End(F^{a-1}/F^a)$, so $
\dim(P/SP)=\sum_{a=1}^{\ell}m_a^2.$
Thus $
\dim [A,P]=\dim SP$. Since $[A,P]\subset SP$, we obtain $[A,P]=SP$.

Now suppose that $\phi\in End(V)$ satisfies $[A,\phi]\in SP$. By the equality just
proved, there exists $\psi\in P$ such that $
[A,p]=[A,\phi]$. Then $\phi-\psi\in Z(A)$. Since $Z(A)\subset P$, we have $\phi-\psi\in P$, and hence $\phi\in P$.
\end{proof}

Take
$x=(E,E_D^\bullet,\Phi)\in \Higgs^{\vec\xi\text{-par}}(C,D,\vec m,\vec\alpha)$
such that
$y=F(x)=(E,\Phi)$ lies in $\Higgs^{\cO,s}(C,D)$. For each $p_i\in D$, set 
\[V_i:=E_{p_i}, \quad A_i:= \res_{p_i}(\Phi),\quad P_i := \PEnd(E)_{p_i}, \quad SP_i:= \SPEnd(E)_{p_i}.\]
There is a natural morphism of complexes
\[
j_x^\bullet:C_{\vxi\text{-par},x}^\bullet\longrightarrow C_{\cO,y}^\bullet
\]
given by the inclusions in degrees $0$ and $1$. In degree $0$, this is the obvious inclusion
$\PEnd(E)\hookrightarrow\End(E)$. In degree $1$, if
$s\in\SPEnd(E)\otimes K_C(D)$, then at each marked point $p_i$ one has
$\res_{p_i}(s)\in SP_i$. By Lemma~\ref{lem:local-linear-algebra-injective},
\[SP_i = [A_i,P_i] \subset [A_i,End(V_i)]\subset T_{A_i}\cO_i\]
for all $i$, so that $s\in\cK_{\cO,y}$. Therefore $j_x^\bullet$ is a morphism of complexes. By taking the hypercohomology $\HH^1(-)$, it induces the differential of $F$, $dF_x=\HH^1(j_x^\bullet)$. 

\begin{proposition}\label{prop:dFx-isomorphism}
The differential
\[
dF_x:\HH^1(C_{\vxi\mathrm{-par},x}^\bullet)\longrightarrow
\HH^1(C_{\cO,y}^\bullet)
\]
is an isomorphism. Moreover, it preserves the symplectic forms; $F^*\omega^{\cO}=\omega^{\vxi\text{-par}}$.
\end{proposition}

\begin{proof}
It is enough to show that the hypercohomology of the cokernel complex
$\Coker(j_x^\bullet)$ vanishes. This cokernel complex is supported on $D$ and is
given locally by
\[
\End(E)/\PEnd(E)
\xrightarrow{[\Phi,-]}
\mathcal K_{\cO,y}/\SPEnd(E)
\]
in degrees $0$ and $1$. Thus it suffices to show that, at each marked point
$p_i$, the induced map
\[ End(V_i)/P_i \xrightarrow{[A_i,-]} [A_i,End(V_i)]/SP_i. \]

It remains to check that the isomorphism preserves the symplectic forms. Recall that the symplectic forms $\omega^{\cO}$ and $\omega^{\vxi\text{-par}}$ are induced by the maps
\[\theta_{\cO,y}:D(C_{\cO,y}^\bullet)\to C_{\cO,y}^\bullet, \quad \theta_{\vxi\text{-par},x}:D(C_{\vxi\text{-par},x}^\bullet)\to C_{\vxi\text{-par},x}^\bullet\]
described in \eqref{e:O serredual map} and \eqref{e:xipar serredual map} respectively. It can be checked easily that the following diagram of complexes commutes:
\[
\begin{tikzcd}
    D(C_{\cO,y}^\bullet) \arrow[r, "D(j_x^\bullet)"] \arrow[d,"\theta_{\cO,y}"] & D(C_{\vxi\text{-par},x}^\bullet) \arrow[d, "\theta_{\vxi\text{-par},x}"] \\
    C_{\cO,y}^\bullet& C_{\vxi\text{-par},x}^\bullet \arrow[l, "j_x^\bullet"]
\end{tikzcd}
\]
where $D(j_x^\bullet)$ is induced from taking the Serre-dual of $j_x^{\bullet}$. Passing to $\bb{H}^1$ yields $F^*\omega^{\cO}=\omega^{\vxi\text{-par}}$. 

\end{proof}

\begin{theorem}\label{thm:symplectic resolution}
        Assume that $\Higgs^{\cO, s}(C,D)\neq \emptyset$ and $F$ is proper and surjective. 
        Then the normalization of $\Higgs^{\overline{\cO}, ss}(C,D)$ has symplectic singularities and the forgetful map $F$ yields a symplectic resolution. Moreover, this resolution is compatible with the Hitchin maps. 
\end{theorem}
\begin{proof}
    Let $U=\Higgs^{\cO, s}(C,D),X = \Higgs^{\overline{\cO}, ss}(C,D)$ and $Y= \Higgs^{\vxi\textrm{-par}}(C,D,\vm, \va)$. Since $F$ is surjective by assumption and $Y$ is irreducible by Proposition~\ref{prop:connected}, $X$ is irreducible. Since $Y$ is smooth, $F$ factors through the normalization $X^\nu$ of $X$: 
    \[ Y\xlongrightarrow{\widetilde{F}}X^{\nu}\to X\]
     Moreover, $\widetilde{F}$ is proper because $F$ is proper, and $\widetilde{F}$ is birational because $F$ is an isomorphism over the dense open $U\subset X$ by Proposition~\ref{prop:F properties}.

    Since $U$ lies in the regular locus $X_{\text{reg}}$ of $X$ by Proposition~\ref{prop:smooth-HiggsO}, we can also regard $U$ as a dense open subset of $V:=(X^{\nu})_{\text{reg}}$. By Beauville's definition of symplectic singularities \cite[Definition 1.1 and Remark 1.2]{Beauville2000}, it suffices to check that the symplectic form $\omega^\cO$ on $U$ extends to $V$. Set $W = \widetilde{F}^{-1}(V)$ and $f:\widetilde{F}|_W:W\to V$. Then $f$ is a proper birational morphism between smooth varieties. Since $f^*\omega^{\cO} = \omega^{\vxi\textrm{-par}}|_{\widetilde{F}^{-1}(U)}$ is holomorphic on $W$, the fact that a meromorphic 2-form $\omega^\cO$ on $V$ is holomorphic if and only if $f^*\omega^{\cO}$ is holomorphic implies that $\omega^\cO$ extends to a holomorphic 2-form $\omega_V$ on $V$. Moreover, since $f^*\omega_V = \omega^{\vxi\textrm{-par}}|_W$ is nondegenerate, the differential $df$ has trivial kernel. Then $\dim(W)= \dim(V)$ implies that $df$ is an isomorphism, so $\omega_V$ is nondegenerate everywhere on $V$. Finally, since $d\omega_V$ vanishes on the dense open subset $U$, it vanishes on $V$. Therefore, we see that $\omega_V$ is a holomorphic symplectic form on $V$ extending $\omega^\cO$. Hence, $X^\nu$ has symplectic singularities and $\widetilde{F}:Y\to X^\nu$ is a symplectic resolution. 

\end{proof}

\begin{rem}
   Let $C$ be an elliptic curve with origin $o$, and let $r \geq 1$. Consider the cotangent bundle $T^*C \cong C \times \mathbb{A}^1$, and let $\pr:T^*C \to \mathbb{A}^1$ be the projection to the fiber coordinate. Taking the $r$-th symmetric product, we obtain a morphism $\pr^r:\mathrm{Sym}^r(T^*C)\to \mathbb{A}^r$ induced by the product of $\pr$. By \cite{Fogarty1968, Beauville2000}, $\mathrm{Sym}^r(T^*C)$ admits a symplectic resolution $\pi:\mathrm{Hilb}^r(T^*C)\to \mathrm{Sym}^r(T^*C)$. Thus we have a commutative diagram
\begin{equation*}\label{eq:diag Hilb}
    \begin{tikzcd}
\mathrm{Hilb}^r(T^*C) \arrow[r, "\pi"] \arrow[rd, "\pr^r\circ\pi"'] &
\mathrm{Sym}^r(T^*C) \arrow[d, "\pr^r"] \\
& \mathbb{A}^r.
\end{tikzcd}
\end{equation*}
On the other hand, take $D=o$, $\vec m=(r-1,1)$, and $\vec \xi=(0,0)$, and let $\cO$ be the nilpotent orbit of Jordan type $(2,1,\dots,1)$, namely with one Jordan block of size $2$ and hence $(r-1)$ Jordan blocks in total. Then we have the commutative diagram 
\[
\begin{tikzcd}
\Higgs^{\vec\xi\text{-par}}(C,D,\vec m,\alpha) \arrow[r, "F"] \arrow[rd, swap, "h(\vm)_{\vxi}"] &
\Higgs^{\overline{\cO},ss}(C,D) \arrow[d, "h^{\ov\cO}"] \\
& B(\vec m)_{\vec\xi}.
\end{tikzcd}
\]
Here, the Hitchin base becomes 
\[B(\vec m)_{\vec \xi}\cong H^0(C,K_C)\oplus \bigoplus_{i=1}^{r-1} H^0\left(C,\left(K_C\left(D\right)\right)^{\otimes i+1}(-iD)\right)\cong \mathbb{A}^r,\]
see \cite{Suwangwen2022}. By \cite[Theorem 5.1]{Groechenig2014HilbertHiggs}, there is an isomorphism $\mathrm{Hilb}^r(T^*C) \cong \Higgs^{\vxi\text{-par}}(C,D,\vm,\va)$ under which $\pr^r\circ \pi$ is compatible with $h(\vm)_{\vxi}$. Recently, Jia showed that this isomorphism becomes a symplectomorphism \cite[Theorem 3.5]{Jia2025Symplectic}. It is a natural question to ask how these two diagrams are related; see also  \cite[Theorem 3.6]{Jia2025Symplectic}.
\end{rem}

\begin{proof}[Proof of Theorem~\ref{thm:connected HiggsO}]
Since $\Higgs^{\vec\xi\text{-par}}(C,D,\vec m,\vec\alpha)$ is smooth and connected, it is irreducible. Hence $\operatorname{Im}(F)$ is irreducible. By Proposition~\ref{prop:F properties}, $
\Higgs^{\ov\cO,s}(C,D)\subseteq \operatorname{Im}(F)$. By Lemma~\ref{lem:residue-loci-closed-locally-closed}, $\Higgs^{\ov\cO,s}(C,D)$ is open in $\Higgs^{\ov\cO,ss}(C,D)$. Since $F$ maps to $\Higgs^{\ov\cO,ss}(C,D)$, the locus $\Higgs^{\ov\cO,s}(C,D)$ is an open subset of $\operatorname{Im}(F)$. It is nonempty because it contains $\Higgs^{\cO,s}(C,D)$. Therefore $\Higgs^{\ov\cO,s}(C,D)$ is irreducible.

Again by Lemma~\ref{lem:residue-loci-closed-locally-closed}, $\Higgs^{\cO,s}(C,D)$ is open in $\Higgs^{\ov\cO,s}(C,D)$. Since it is nonempty by assumption, it is a nonempty open subset of an irreducible space. Hence $\Higgs^{\cO,s}(C,D)$ is irreducible, and in particular connected. Finally, it is dense in $\Higgs^{\ov\cO,s}(C,D)$ and $\Higgs^{\ov\cO,s}(C,D)$ is closed in $\Higgs^s(C,D)$, hence
\[
\overline{\Higgs^{\cO,s}(C,D)}
=
\Higgs^{\ov\cO,s}(C,D).
\]
\end{proof}

\section{Applications to Betti moduli spaces}\label{sec:Betti}

The goal of this section is to explain how the results of this article can be transported to the Betti side via Simpson's tame non-abelian Hodge correspondence. As before, let $C$ be a smooth projective curve of genus $g\geq 0$, and let $D=\{p_1,\ldots,p_n\}\subset C$ be a reduced divisor. We write $C^\circ=C\setminus D$ and fix the rank $r$.

\begin{definition}
	A filtered local system on $C^\circ$ is a local system $\mathbb L$ together with, for each $p_i\in D$, a decreasing left-continuous filtration $\{\mathbb L_{\widetilde p_i}^{\beta}\}_{\beta\in\mathbb R}$ on a nearby stalk $\mathbb L_{\widetilde p_i}$, preserved by the local monodromy $T_i$ around $p_i$. Here $\widetilde p_i$ denotes a chosen nearby point in $C^\circ$.
\end{definition}

There is also an analogous definition of a filtered Higgs bundle, which is equivalent to the definition of a parabolic Higgs bundle $(E,E_D^\bullet,\Phi,\vec\alpha)$ on $C$ used in this paper. See \cite{simpson-noncompact,Yokogawa-compactification,LeeLee2025} for more details.

At each $p_i$, both filtered local systems and parabolic Higgs bundles contain the data of a filtered vector space together with an endomorphism preserving the filtration.

\begin{definition}
	Let $V$ be a finite-dimensional complex vector space equipped with a decreasing left-continuous filtration $\{V^\gamma\}_{\gamma\in\mathbb R}$, and let $T:V\to V$ be an endomorphism preserving the filtration, i.e. $T(V^\gamma)\subseteq V^\gamma$ for all $\gamma$. The associated graded vector space is
	\[
	\operatorname{gr}(V)=\bigoplus_\gamma \operatorname{gr}_\gamma(V),
	\]
	where $\operatorname{gr}_\gamma(V)=V^\gamma/V^{\gamma+\epsilon}$ for sufficiently small $\epsilon>0$. The endomorphism $T$ induces an endomorphism $\operatorname{gr}(T)$ on $\operatorname{gr}(V)$.
\end{definition}

To the data $(V,\{V^\gamma\},T)$, we associate a collection of partitions as follows. A conjugacy class $\mathcal C_0\subset GL_r(\C)$ is determined by its Jordan normal form, so we may identify $\mathcal C_0$ with a collection of partitions $\{P^\nu\}$ labeled by eigenvalues $\nu$, where each partition $P^\nu=(n_1,\ldots,n_\ell)$ records the sizes of the Jordan blocks with eigenvalue $\nu$. Since $\operatorname{gr}(T)$ restricts to an endomorphism on each $\operatorname{gr}_\gamma(V)$, we take the corresponding collection of partitions $\{P^{\gamma,\nu}\}$. We call this collection the residue diagram of $(V,\{V^\gamma\},T)$. For a filtered local system, this construction is applied to the local monodromy $T_i$. For a parabolic Higgs bundle, it is applied to the residue $\res_{p_i}(\Phi)$.

\begin{exmp}\label{exmp:xipar}
	A $\vec\xi$-parabolic Higgs bundle is the same as a parabolic Higgs bundle whose graded residue at each $p_i$ acts by the scalar $\xi_{i,j}$ on the graded piece of dimension $m_{i,j}$. Thus the corresponding residue diagram has $
	P^{\alpha_{i,j},\xi_{i,j}}=(1,\ldots,1)$ of size $m_{i,j}$ for each $i,j$.
\end{exmp}

\begin{exmp}\label{exmp:trivialfiltrations}
	When the filtrations of a filtered local system on $C^\circ$ are trivial with $\beta=0$, the residue diagrams $\{P^{0,\lambda}\}$ describe the conjugacy classes of the local monodromies around the punctures.

	Similarly, if the parabolic Higgs bundle $(E,E_D^\bullet,\Phi,\vec\alpha)$ has trivial parabolic filtrations, i.e. $E_{p_i}^0=E_{p_i}$ and $E_{p_i}^{\epsilon}=0$ for $\epsilon>0$, with a single parabolic weight $\alpha_{i,1}$ at each $p_i$, then the residue diagrams $\{P^{\alpha_{i,1},\xi}\}$ describe the conjugacy classes of $\res_{p_i}(\Phi)$. One may also take the constant weights $\alpha$ and $\beta$ to be nonzero, but the ordinary character-variety case below corresponds to $\beta=0$.
\end{exmp}

Simpson's tame non-abelian Hodge correspondence matches these residue diagrams. If the Dolbeault-side weight is $\alpha$, the Higgs residue eigenvalue is $\xi=b+\sqrt{-1}c$, the Betti-side filtration jump is $\beta$, and the local monodromy eigenvalue is $\lambda$, then the correspondence is given by the following table:
\begin{center}
\begin{table}[h]
    \centering
    \caption{Simpson's table}\label{table:simpson}
    \begin{tabular}{|c|c|c|}
        \hline
         & parabolic Higgs bundle & filtered local system \\
        \hline
        weight/jump & $\alpha$ & $\beta=-2b$ \\
        \hline
        eigenvalue & $\xi=b+\sqrt{-1}c$ & $\lambda=\exp(-2\pi\sqrt{-1}\alpha+4\pi c)$ \\
        \hline
    \end{tabular}
\end{table}
\end{center}
In other words, the residue diagram $\{P^{\alpha,\xi}\}$ on the Dolbeault side is identified with the residue diagram $\{P^{\beta,\lambda}\}$ on the Betti side after the above change of labels.

We now formulate the moduli-level statement in the form used below. Fix, for each $p_i\in D$, parabolic weights $\vec\alpha_i=(\alpha_{i,1},\ldots,\alpha_{i,\ell_i})$, Higgs residue eigenvalues $\vec\xi_i=(\xi_{i,1},\ldots,\xi_{i,\ell_i})$, and multiplicities $\vec m_i=(m_{i,1},\ldots,m_{i,\ell_i})$ with $\sum_jm_{i,j}=r$. Write $\xi_{i,j}=b_{i,j}+\sqrt{-1}c_{i,j}$. The corresponding Betti-side jumps and monodromy eigenvalues are
\[
\beta_{i,j}=-2b_{i,j},
\qquad
\lambda_{i,j}=\exp(-2\pi\sqrt{-1}\alpha_{i,j}+4\pi c_{i,j}).
\]

We denote by $P^{\vec\alpha,\vec\xi}_{\vec m}$ the collection of residue diagrams $\{P^{\alpha_{i,j},\xi_{i,j}}\mid 1\leq i\leq n,\ 1\leq j\leq \ell_i\}$, where the size of the partition $P^{\alpha_{i,j},\xi_{i,j}}$ is $m_{i,j}$. Similarly, we define $P^{\vec\beta,\vec\lambda}_{\vec m}$ on the Betti side.

Let $\mathcal M^{ss}_{\mathrm{Dol}}(P^{\vec\alpha,\vec\xi}_{\vec m})$ be the coarse moduli space of $\vec\alpha$-semistable parabolic Higgs bundles of rank $r$ and parabolic degree zero whose residue diagram is $P^{\vec\alpha,\vec\xi}_{\vec m}$. Let $\mathcal M^{s}_{\mathrm{Dol}}(P^{\vec\alpha,\vec\xi}_{\vec m})$ denote its stable locus. Similarly, let $\mathcal M^{ss}_{\mathrm{Betti}}(P^{\vec\beta,\vec\lambda}_{\vec m})$ be the coarse moduli space of $\vec\beta$-semistable filtered local systems of rank $r$ and filtered degree zero whose residue diagram is $P^{\vec\beta,\vec\lambda}_{\vec m}$, and let $\mathcal M^{s}_{\mathrm{Betti}}(P^{\vec\beta,\vec\lambda}_{\vec m})$ denote its stable locus.

\begin{theorem}[Tame non-abelian Hodge correspondence]\cite{simpson-noncompact,BGPM2020}\label{thm:tame-NAHC}
	With the notation above, Simpson's tame non-abelian Hodge correspondence gives a homeomorphism of coarse moduli spaces
	\[
	\mathcal M^{ss}_{\mathrm{Dol}}(P^{\vec\alpha,\vec\xi}_{\vec m})
	\simeq_{\mathrm{top}}
	\mathcal M^{ss}_{\mathrm{Betti}}(P^{\vec\beta,\vec\lambda}_{\vec m}),
	\]
	and it restricts to a homeomorphism of stable loci
	\[
	\mathcal M^{s}_{\mathrm{Dol}}(P^{\vec\alpha,\vec\xi}_{\vec m})
	\simeq_{\mathrm{top}}
	\mathcal M^{s}_{\mathrm{Betti}}(P^{\vec\beta,\vec\lambda}_{\vec m}),
	\]
	when the local data on both sides are related by Simpson's table, namely
	$\beta_{i,j}=-2\operatorname{Re}(\xi_{i,j})$ and
	$\lambda_{i,j}=\exp(-2\pi\sqrt{-1}\alpha_{i,j}+4\pi\operatorname{Im}(\xi_{i,j}))$.
\end{theorem}

In the main part of this article, we studied moduli spaces of $\vec\xi$-parabolic Higgs bundles and meromorphic Higgs bundles with fixed residue conjugacy classes. Both can be regarded as Dolbeault moduli spaces with prescribed residue diagrams as described in Examples \ref{exmp:xipar} and \ref{exmp:trivialfiltrations}.

First, consider the case where the residue diagram $P^{\vec\alpha,\vec\xi}_{\vec m}$ is of the type appearing in Example \ref{exmp:xipar}: on each graded piece, the residue acts by a prescribed scalar. Then the Dolbeault moduli space $\mathcal M^{ss}_{\mathrm{Dol}}(P^{\vec\alpha,\vec\xi}_{\vec m})$ is precisely the moduli space $\Higgs^{\vec\xi\text{-par},ss}(C,D,\vec m,\vec\alpha)$ of $\vec\alpha$-semistable $\vec\xi$-parabolic Higgs bundles of parabolic degree zero. The stable loci match as well:
\[
\mathcal M^{s}_{\mathrm{Dol}}(P^{\vec\alpha,\vec\xi}_{\vec m})
=
\Higgs^{\vec\xi\text{-par},s}(C,D,\vec m,\vec\alpha).
\]
By Theorem \ref{thm:tame-NAHC}, the corresponding stable Betti moduli space $\mathcal M^{s}_{\mathrm{Betti}}(P^{\vec\beta,\vec\lambda}_{\vec m})$ is homeomorphic to $\Higgs^{\vec\xi\text{-par},s}(C,D,\vec m,\vec\alpha)$.

\begin{theorem}\label{thm:connected-Betti-xi-par}
	Suppose that the Betti residue diagram $P^{\vec\beta,\vec\lambda}_{\vec m}$ corresponds, via Simpson's table, to scalar graded residue data of type $(\vec m,\vec\xi)$ on the Dolbeault side. Then the stable Betti moduli space
	$\mathcal M^{s}_{\mathrm{Betti}}(P^{\vec\beta,\vec\lambda}_{\vec m})$
	is connected whenever it is nonempty.
\end{theorem}

\begin{proof}
	Given such a residue diagram $P^{\vec\beta,\vec\lambda}_{\vec m}$, take the corresponding residue diagram $P^{\vec\alpha,\vec\xi}_{\vec m}$ on the Dolbeault side. Then
	\[
	\mathcal M^{s}_{\mathrm{Dol}}(P^{\vec\alpha,\vec\xi}_{\vec m})
	=
	\Higgs^{\vec\xi\text{-par},s}(C,D,\vec m,\vec\alpha).
	\]
	The connectedness follows from Proposition~\ref{prop:connected} and the homeomorphism of Theorem \ref{thm:tame-NAHC}.
\end{proof}

Next, we consider a trivial parabolic filtration, meaning that $E_{p_i}^0=E_{p_i}$ and $E_{p_i}^{\epsilon}=0$ for $\epsilon>0$, with a single parabolic weight $\alpha_{i,1}\in[0,1)$ at each puncture. The residue diagram has only one jump at each puncture, and it is simply the Jordan diagram of the residue $\res_{p_i}(\Phi)$. We write $\cO_i\subset\mathfrak{gl}_r(\C)$ for the conjugacy class whose residue diagram is $P^{\alpha_{i,1},\vec\xi_i}$. Automatically, $\mathcal{O}=\prod_{i=1}^n\mathcal{O}_i$ satisfies the residue condition \eqref{eq:trace-compatibility}. In this case, $\vec\alpha$-(semi)stability is equivalent to the usual slope (semi)stability for the underlying meromorphic Higgs bundle. Hence
\[
\mathcal M^{ss}_{\mathrm{Dol}}(P^{\vec\alpha,\vec\xi}_{\vec m})
\cong
\Higgs^{\cO,ss}(C,D),
\qquad
\mathcal M^{s}_{\mathrm{Dol}}(P^{\vec\alpha,\vec\xi}_{\vec m})
\cong
\Higgs^{\cO,s}(C,D).
\]

On the Betti side, we take the trivial filtration with the single jump $\beta_{i,1}=0$ at each puncture. Then the residue diagram has only one jump at each puncture, and it is simply the Jordan diagram of the local monodromy $T_i$. We write $\mathcal C_i\subset GL_r(\C)$ for the conjugacy class whose residue diagram is $P^{0,\vec\lambda_i}$. Let $\mathcal{C}=(\mathcal C_1,\dots,\mathcal C_n)$ be a collection of conjugacy
classes. Analogously to the residue condition~\eqref{eq:trace-compatibility}, we
impose the condition
\begin{equation}\label{eq:trace-compatibility Betti}
    \prod_{i=1}^n\det(\mathcal C_i)=1.
\end{equation}
We recall the relation with character varieties. For such $\mathcal{C}$, define the representation variety
\[
R_{\mathcal C}
=
\left\{
(A_1,B_1,\ldots,A_g,B_g,M_1,\ldots,M_n)
\in GL_r(\C)^{2g}\times\prod_{i=1}^n\mathcal C_i
\ \middle|\
\prod_{j=1}^g[A_j,B_j]\prod_{i=1}^nM_i=I
\right\}.
\]
The group $GL_r(\C)$ acts on $R_{\mathcal C}$ by simultaneous conjugation. The associated character variety is the affine GIT quotient $\chi(C^\circ,\mathcal C):=R_{\mathcal C}//GL_r(\C)$. We denote its irreducible locus by $\chi(C^\circ,\mathcal C)^{\operatorname{irr}}$.

Since $\beta=0$, the filtered degree on the Betti side is zero: the underlying flat bundle has degree zero and the weight contribution is zero. Hence every $0$-filtered local system is semistable. It is stable if and only if the underlying representation of $\pi_1(C^\circ)$ is irreducible, and it is polystable if and only if the representation is semisimple. Thus we have the dictionary
\begin{center}
	\begin{tabular}{|c|c|}
		\hline
		$0$-filtered local system & representation-theoretic meaning \\
		\hline
		semistable & arbitrary representation \\
		\hline
		polystable & semisimple representation \\
		\hline
		stable & irreducible representation \\
		\hline
	\end{tabular}
\end{center}
Since the affine GIT quotient parametrizes closed orbits, equivalently semisimple representations, or polystable $0$-filtered local systems, we have the following identifications on the Betti side:
\[
\mathcal M^{ss}_{\mathrm{Betti}}(P^{\vec\beta,\vec\lambda}_{\vec m})
\cong
\chi(C^\circ,\mathcal C),
\qquad
\mathcal M^{s}_{\mathrm{Betti}}(P^{\vec\beta,\vec\lambda}_{\vec m})
\cong
\chi(C^\circ,\mathcal C)^{\mathrm{irr}}.
\]
where $P^{\vec\beta,\vec\lambda}_{\vec m}$ is the associated residue diagram to $\mathcal{C}$.

We now impose compatibility between the Dolbeault and Betti local data. We take $\vec\alpha_i=\alpha_{i,1}\in[0,1)$ for all $i$ and take the Betti filtration to be trivial, namely $\vec\beta=0$. We also impose the parabolic degree zero condition, that is $d+\sum_{i=1}^n\alpha_{i,1}r=0$. By Simpson's table, the Higgs residue eigenvalues must be purely imaginary, say $\xi_{i,j}=\sqrt{-1}c_{i,j}$, and the corresponding local monodromy eigenvalues are $
\lambda_{i,j}=\exp(-2\pi\sqrt{-1}\alpha_{i,1}+4\pi c_{i,j})$.

We take the corresponding residue conjugacy classes $\cO_i\subset\mathfrak{gl}_r(\C)$ and monodromy conjugacy classes $\mathcal C_i\subset GL_r(\C)$. Applying Theorem \ref{thm:tame-NAHC} in this case, we obtain homeomorphisms of coarse moduli spaces
\begin{equation}\label{e:higgsOchar}
    \Higgs^{\cO,ss}(C,D)\simeq_{\mathrm{top}}\chi(C^\circ,\mathcal C),
    \qquad
    \Higgs^{\cO,s}(C,D)\simeq_{\mathrm{top}}\chi(C^\circ,\mathcal C)^{\operatorname{irr}}.
\end{equation}

For the application, we introduce one more notation for $\mathcal{C}$. We say $\mathcal{C}$ is $q$-divisible if there exists $m \geq 2$ such that $\mathcal{C} \sim m\mathcal{C}'$ for some collection of conjugacy classes $\mathcal{C}'$. We also say $\mathcal{C}$ is $q$-indivisible if $\mathcal{C}$ is not $q$-divisible. 

\begin{theorem}\label{thm:connected-character-variety}
	Let $\mathcal C=(\mathcal C_1,\ldots,\mathcal C_n)$ be $q$-indivisible such that $\prod_{i=1}^n\det(\mathcal C_i)=1.$ Assume that, for each $i$, all the eigenvalues of $\mathcal C_i$ have a common argument. Then the irreducible locus
	$\chi(C^\circ,\mathcal C)^{\operatorname{irr}}$
	is connected whenever it is nonempty. 
\end{theorem}

\begin{proof}
	Under the assumptions, the conjugacy classes $\mathcal C_i$ are compatible, via Simpson's table with $\beta=0$, with a collection of purely imaginary residue conjugacy classes $\cO_i$ on the Dolbeault side satisfying the parabolic degree condition. By Remark \ref{rem:assump}, the $q$-indivisible condition implies that Assumption \ref{assmp:genericity} holds. Therefore \eqref{e:higgsOchar} gives a homeomorphism
	\[
	\Higgs^{\cO,s}(C,D)\simeq_{\mathrm{top}}
	\chi(C^\circ,\mathcal C)^{\operatorname{irr}}.
	\]
	The connectedness of $\Higgs^{\cO,s}(C,D)$ from Theorem~\ref{thm:connected HiggsO} implies the claim.
\end{proof}

We finally comment on orbit closures. Let $\overline{\cO_i}$ and $\overline{\mathcal C_i}$ be the Zariski closures of compatible conjugacy classes. The condition $\res_{p_i}(\Phi)\in\overline{\cO_i}$ allows all residue diagrams occurring in the closure order of $\cO_i$. Similarly, the condition $M_i\in\overline{\mathcal C_i}$ allows the corresponding monodromy diagrams in the closure order of $\mathcal C_i$. Since the tame non-abelian Hodge correspondence preserves residue diagrams after relabelling by Simpson's table, it matches these strata one by one.

Thus one obtains a natural stratumwise correspondence between $\Higgs^{\overline{\cO},s}(C,D)$ and $\chi(C^\circ,\overline{\mathcal C})^{\mathrm{irr}}$. However, to the author's knowledge, a homeomorphism between these two spaces, equipped with the subspace topologies coming from the orbit-closure conditions, does not seem to be explicitly stated in the literature. In particular, the passage from a stratumwise correspondence to a homeomorphism of the closure loci requires compatibility with limits across strata.

\begin{corollary}
    Suppose that the stratumwise tame non-abelian Hodge correspondence for the compatible orbit closures extends to a homeomorphism
    \[
    \Higgs^{\overline{\cO},s}(C,D)
    \simeq_{\mathrm{top}}
    \chi(C^\circ,\overline{\mathcal C})^{\mathrm{irr}}.
    \]
    If the corresponding open residue-orbit locus $\Higgs^{\cO,s}(C,D)$ is nonempty, then
    $\chi(C^\circ,\overline{\mathcal C})^{\mathrm{irr}}$ is connected.
\end{corollary}

\begin{proof}
    By Theorem~\ref{thm:connected HiggsO}, the nonemptiness of $\Higgs^{\cO,s}(C,D)$ implies that $\Higgs^{\overline{\cO},s}(C,D)$ is irreducible, hence connected. The claim follows from the assumed homeomorphism.
\end{proof}

\appendix
\section{Conjugacy classes and compatible flags}\label{sec:local-linear-algebra}
In this section, we collect some facts needed to relate conjugacy classes to compatible parabolic flags. The analogous results for nilpotent orbits can be found in \cite[Section 3]{BorhoMacPherson1983}. We include the direct proofs in our setting for completeness.

Let $V$ be a complex vector space of dimension
$r$. Fix a partition $\um=(m_1,\ldots,m_\ell)$ of $r$ and a tuple
$\uxi=(\xi_1,\ldots,\xi_\ell)\in\C^\ell$. A partial flag of type $\um$ is a chain
$0=F^\ell\subset F^{\ell-1}\subset\cdots\subset F^1\subset F^0=V$ such that
$\dim(F^{a-1}/F^a)=m_a$ for $1\leq a\leq \ell$. We denote the corresponding
partial flag variety by $\Flag_{\um}(V)$.

\begin{definition}\label{def:compatible}
	An endomorphism $A\in End(V)$ is \emph{compatible with $(\um,\uxi)$} if there
	exists a flag $F^\bullet\in\Flag_{\um}(V)$ such that
	\[
	A(F^a)\subset F^a,
	\qquad
	(A-\xi_a\id)(F^{a-1})\subset F^a
	\quad\text{for all }1\leq a\leq \ell.
	\]
	Equivalently, the induced endomorphism on each graded piece is scalar:
	\[\gr_a(A)=\xi_a\id_{F^{a-1}/F^a}.\]
	 Such a flag is called a
	\emph{compatible flag} for $A$.
\end{definition}

\noindent Define the incidence variety
\[
\widetilde X_{\um,\uxi}(V)
:=
\left\{
(A,F^\bullet)\in End(V)\times \Flag_{\um}(V)
\ \middle|\
(A-\xi_a\id)(F^{a-1})\subset F^a\text{ for all }a
\right\},
\]
and let $\mu_{\um,\uxi}:\widetilde X_{\um,\uxi}(V)\to End(V)$ be the projection
$(A,F^\bullet)\mapsto A$. We write
$X_{\um,\uxi}(V):=\mu_{\um,\uxi}(\widetilde X_{\um,\uxi}(V))$.

\begin{proposition}\label{prop:X-closed-irreducible}
	The variety $\widetilde X_{\um,\uxi}(V)$ is smooth and irreducible, and the morphism
	$\mu_{\um,\uxi}$ is projective. Moreover, $X_{\um,\uxi}(V)$ is a closed irreducible
	$\GL(V)$-stable subvariety of $End(V)$.
\end{proposition}

\begin{proof}
  Let $\mathrm{pr}:\widetilde X_{\um,\uxi}(V)\to \Flag_{\um}(V)$ be the projection.
    For a fixed flag $F^\bullet$, the fiber of $\mathrm{pr}$ is the affine linear subspace
    of $End(V)$ defined by
    $(A-\xi_a\mathrm{id})(F^{a-1})\subset F^{a}$ for all $a$. Its translation
    vector space is
    \[
    \{B\in End(V)\mid B(F^{a-1})\subset F^a\text{ for all }a\},
    \]
    whose dimension depends only on the type $\um$. Hence
    $\widetilde X_{\um,\uxi}(V)$ is an affine bundle over the smooth irreducible
    variety $\Flag_{\um}(V)$. In particular, $\widetilde X_{\um,\uxi}(V)$ is smooth
    and irreducible.

    The morphism $\mu_{\um,\uxi}$ is projective because it is the restriction of
    the projection $End(V)\times \Flag_{\um}(V)\to End(V)$, and
    $\Flag_{\um}(V)$ is projective. Hence its image
    $X_{\um,\uxi}(V)$ is closed. Since $\widetilde X_{\um,\uxi}(V)$ is irreducible,
    its image is irreducible. The $\GL(V)$-stability is immediate from the
    definition.
\end{proof}

We now explain how conjugacy classes are encoded by pairs $(\um,\uxi)$. Let
$\mathcal O\subset End(V)$ be a conjugacy class and choose $A\in\mathcal O$. For
each eigenvalue $\lambda$, write $V_\lambda$ for the generalized eigenspace and
$A|_{V_\lambda}=\lambda I+N_\lambda$, where $N_\lambda$ is nilpotent. If the
Jordan type of $N_\lambda$ is the partition $\underline{\mu}^{(\lambda)}$, write
$(\underline{\mu}^{(\lambda)})^t=(m_{\lambda,1},\ldots,m_{\lambda,\ell_\lambda})$ for its
conjugate partition. Concatenating these sequences over all eigenvalues, and
repeating each eigenvalue $\lambda$ exactly $\ell_\lambda$ times, gives a pair
$(\um,\uxi)$.
\begin{theorem}\label{thm:orbit-correspondence}
    Let $\cO$ be a conjugacy class and let $(\um, \uxi)$ be the associated pair as above. 
    \begin{enumerate}
        \item $X_{\um,\uxi}(V)=\overline{\mathcal O}$. In particular, an endomorphism $A\in End(V)$ is compatible with $(\um,\uxi)$ (admitting a compatible flag)
    if and only if $A\in \overline{\mathcal O}$.
        \item For any $A\in \cO$, the compatible flag is unique.
    \end{enumerate}
\end{theorem}

\begin{proof}
Let $A\in\mathcal O$. For each eigenvalue $\lambda$, let $V_\lambda$ be the generalized $\lambda$-eigenspace and set $N_\lambda=(A-\lambda\id)|_{V_\lambda}$. Consider the image filtration $F^\bullet_\lambda: V_\lambda\supset N_\lambda V_\lambda\supset N_\lambda^2V_\lambda\supset\cdots$, which clearly satisfies the compatibility condition $(A-\lambda \id)(F^{a-1}) \subset F^a$ for all $a.$ Moreover, if the Jordan type of $N_\lambda$ is $\underline{\mu}^{(\lambda)}$, then the successive quotients of the image filtration have dimensions given by the conjugate partition $(\underline{\mu}^{(\lambda)})^t$ (see e.g. \cite[Remark 2.4]{LeeLee2025}). After choosing an ordering of the eigenvalue data, these filtrations give a compatible flag of type $(\um,\uxi)$. Hence $\mathcal O\subset X_{\um,\uxi}(V)$. 

Conversely, let $B\in X_{\um,\uxi}(V)$ and choose a compatible flag $F^\bullet$. The compatibility condition implies that the flag $F^\bullet$ must respect the generalized eigenspace decomposition $V=\bigoplus V_{\lambda}$ of $B$. Fix an eigenvalue $\lambda$, let $N_\lambda=(B-\lambda\id)|_{V_\lambda}$ be the nilpotent part and $F^\bullet_\lambda$ be the induced filtration of $V_\lambda$. Then the compatibility condition $N_\lambda(F^{a-1}_\lambda)\subset F^a_\lambda$ implies that $N_\lambda^aV_\lambda\subset F^a_\lambda$, and hence, 
\begin{equation}\label{eq:rank ineq}
\dim(N_\lambda^aV_\lambda)\leq \dim (F^a_\lambda), \quad \textrm{for all }a.
\end{equation}
Let $\underline{\nu}^{(\lambda)}$ be the Jordan type of $N_\lambda$. Argue as in the preceding paragraph, we have $\dim(N_\lambda^aV_\lambda) = \dim (V_\lambda) - \sum_{i=1}^a (\underline{\nu}^{(\lambda)})_i^t$ and $\dim(F^a_\lambda) = \dim(V_\lambda) - \sum_{i=1}^a (\underline{\mu}^{(\lambda)})_i^t$. Substituting them into the inequality \ref{eq:rank ineq} yields: 
\begin{equation*}
    \sum_{i=1}^a (\underline{\nu}^{(\lambda)})^t_i \geq \sum_{i=1}^a (\underline{\mu}^{(\lambda)})^t_i, \quad \textrm{for all }a.
\end{equation*}
This means that $(\underline{\nu}^{(\lambda)})^t\geq (\underline{\mu}^{(\lambda)})^t$. By taking conjugation, we have $(\underline{\nu}^{(\lambda)})\leq (\underline{\mu}^{(\lambda)})$. Therefore, $N_\lambda$ lies in the closure of the nilpotent orbit of type $\underline{\mu}^{(\lambda)}$ for every $\lambda$. It follows that the conjugacy class of $B$ is contained in $\overline{\mathcal O}$. Thus $X_{\um,\uxi}(V)\subset \overline{\mathcal O}$. Since $\mathcal O\subset X_{\um,\uxi}(V)$ and $X_{\um,\uxi}(V)$ is closed by Proposition~\ref{prop:X-closed-irreducible}, we obtain $X_{\um,\uxi}(V)=\overline{\mathcal O}$. 

Finally, we establish uniqueness. Suppose $A\in \cO$ and $F^\bullet$ is any compatible flag for $A$. Applying the argument from the previous paragraph, the compatibility condition implies that $N_\lambda^aV_\lambda\subset F^a_\lambda$ for all $a$. Since $A\in \cO$, the Jordan type of its nilpotent part $N_\lambda$ is exactly $\underline{\mu}^{(\lambda)}$. So, the inequality \ref{eq:rank ineq} becomes an equality: $\dim(N_\lambda^aV_\lambda) = \dim (V_\lambda) - \sum_{i=1}^a (\underline{\mu}^{(\lambda)})_i^t = \dim(F^a_\lambda)$. Therefore, we have $ N_\lambda^aV_\lambda=F^a_\lambda $ for all $a$, which means that the compatible flag is exactly the image filtration and is uniquely determined by $A$.
\end{proof}
\begin{remark}
	The ordering of the entries of $(\um,\uxi)$ is irrelevant up to simultaneous
	permutation. Thus the essential data are the eigenvalues together with, for
	each eigenvalue, the conjugate partition of the nilpotent Jordan type on the
	corresponding generalized eigenspace.
\end{remark}

\begin{example}\label{ex:2111}
	Consider the conjugacy class with Jordan form
	\[
	J_2(1)\oplus J_1(1)\oplus J_2(2),
	\]
	where $J_a(b)$ denotes the Jordan block of size $a$ with eigenvalue $b$.
	For the eigenvalue $1$, the Jordan type is $(2,1)$, whose transpose is again
	$(2,1)$. For the eigenvalue $2$, the Jordan type is $(2)$, whose transpose is
	$(1,1)$. Therefore the associated pair is
	$\um=(2,1,1,1)$ and $\uxi=(1,1,2,2)$, up to simultaneous permutation of the
	indices.
	
	Let $V_1$ and $V_2$ be the generalized eigenspaces for the eigenvalues $1$
	and $2$, respectively, and set $N_i=(A-i\id)|_{V_i}$. With the ordering
	$\uxi=(1,1,2,2)$, the compatible flag is given by
	\[
	F^0=V_1\oplus V_2,\qquad
	F^1=N_1V_1\oplus V_2,\qquad
	F^2=V_2,\qquad
	F^3=N_2V_2,\qquad
	F^4=0.
	\]
	Then $\dim F^{a-1}/F^a=m_a$, and one checks directly that
	$(A-\xi_a\id)(F^{a-1})\subset F^a$ for every $a$. Thus this is the compatible
	flag of type $(\um,\uxi)$. By Theorem~\ref{thm:orbit-correspondence}, the
	variety $X_{\um,\uxi}(V)$ is exactly the closure of this conjugacy class.
\end{example}

\section*{Acknowledgement}
\noindent We would like to thank Ron Donagi, Andres Fernandez Herrero, and Cheng Shu for helpful discussions.
S.L. was supported by the Institute for Basic Science (IBS-R003-D1).

\printbibliography

\end{document}